\documentclass[times]{speauth}

\usepackage{moreverb}

\usepackage{amsfonts,amssymb,amsthm,newlfont,graphicx,subfigure,algorithmic,algorithm,fixltx2e,booktabs,bm}
\usepackage{epstopdf}
\usepackage{empheq} 
\usepackage[colorlinks,bookmarksopen,bookmarksnumbered,citecolor=red,urlcolor=red]{hyperref}
\numberwithin{algorithm}{section}

\newcommand\BibTeX{{\rmfamily B\kern-.05em \textsc{i\kern-.025em b}\kern-.08em
T\kern-.1667em\lower.7ex\hbox{E}\kern-.125emX}}

\newtheorem{thm}{Theorem}[section]

\newtheorem{rem}{Remark}[section]

\numberwithin{equation}{section}
\renewcommand{\theequation}{\thesection.\arabic{equation}}
\def\simgt{\,\hbox{\lower0.6ex\hbox{$>$}\llap{\raise0.3ex\hbox{$\sim$}}}\,}
\def\simlt{\,\hbox{\lower0.6ex\hbox{$<$}\llap{\raise0.3ex\hbox{$\sim$}}}\,}
\def\simgteq{\,\hbox{\lower0.6ex\hbox{$\ge$}\llap{\raise0.6ex\hbox{$\sim$}}}\,}
\def\simlteq{\,\hbox{\lower0.6ex\hbox{$\le$}\llap{\raise0.6ex\hbox{$\sim$}}}\,}
\makeatletter
\def\user@resume{resume}
\def\user@intermezzo{intermezzo}
\newcounter{previousequation}
\newcounter{lastsubequation}
\newcounter{savedparentequation}
\renewenvironment{subequations}[1][]{%
      \def\user@decides{#1}%
      \setcounter{previousequation}{\value{equation}}%
      \ifx\user@decides\user@resume
           \setcounter{equation}{\value{savedparentequation}}%
      \else
      \ifx\user@decides\user@intermezzo
           \refstepcounter{equation}%
      \else
           \setcounter{lastsubequation}{0}%
           \refstepcounter{equation}%
      \fi\fi
      \protected@edef\theHparentequation{%
          \@ifundefined {theHequation}\theequation \theHequation}%
      \protected@edef\theparentequation{\theequation}%
      \setcounter{parentequation}{\value{equation}}%
      \ifx\user@decides\user@resume
           \setcounter{equation}{\value{lastsubequation}}%
         \else
           \setcounter{equation}{0}%
      \fi
      \def\theequation  {\theparentequation  \alph{equation}}%
      \def\theHequation {\theHparentequation \alph{equation}}%
      \ignorespaces
}{%
  \ifx\user@decides\user@resume
       \setcounter{lastsubequation}{\value{equation}}%
       \setcounter{equation}{\value{previousequation}}%
  \else
  \ifx\user@decides\user@intermezzo
       \setcounter{equation}{\value{parentequation}}%
  \else
       \setcounter{lastsubequation}{\value{equation}}%
       \setcounter{savedparentequation}{\value{parentequation}}%
       \setcounter{equation}{\value{parentequation}}%
  \fi\fi
  \ignorespacesafterend
}
\makeatother

\begin{document}

\runningheads{Kareem T. Elgindy}{High-Order Adaptive Gegenbauer Integral Spectral Element Method}

\title{High-Order Adaptive Gegenbauer Integral Spectral Element Method for Solving Nonlinear Optimal Control Problems}

\author{Kareem T. Elgindy\corrauth}

\address{Mathematics Department, Faculty of Science, Assiut University, Assiut 71516, Egypt}

\corraddr{Mathematics Department, Faculty of Science, Assiut University, Assiut 71516, Egypt.\\Email: kareem.elgindy@aun.edu.eg; kareem.elgindy@gmail.com.}

\begin{abstract}
In this work, we propose an adaptive spectral element algorithm for solving nonlinear optimal control problems. The method employs orthogonal collocation at the shifted Gegenbauer-Gauss points combined with very accurate and stable numerical quadratures to fully discretize the multiple-phase integral form of the optimal control problem. The proposed algorithm relies on exploiting the underlying smoothness properties of the solutions for computing approximate solutions efficiently. In particular, the method brackets discontinuities and ``points of nonsmoothness'' through a novel local adaptive algorithm, which achieves a desired accuracy on the discrete dynamical system equations by adjusting both the mesh size and the degree of the approximating polynomials. A rigorous error analysis of the developed numerical quadratures is presented. Finally, the efficiency of the proposed method is demonstrated on three test examples from the open literature.
\end{abstract}


\maketitle

\vspace{-6pt}

\section{Introduction}
\label{int}
Optimal control theory has become one of the most dominant and indispensable techniques for analyzing dynamical systems in which optimal decisions are sought at each moment. Surely, the principal part in the establishment of the theory as an important and rich area of applied mathematics arises in the strong utilization of the subject area in a great breadth of applications and research areas such as engineering, computer science, astronautics, biological sciences, chemistry, agriculture, business, management, energy, path planning problems, and a host of many other areas; cf. \cite{Fatmawati2015,Mau2016,Kapur2012,Zheng2012,Chen2016,Hung2016,WANG2015,Peng2016,Elgindy2012c,Elgindy2013b}. 

The most popular analytical methods for solving optimal control problems such as the calculus of variations, Pontryagin's principle, and Bellman's principle, can generally solve only fairly simple problems. However, such methods are largely deficient to handle the increasing complexity of optimal control problems since the advent of digital computers, which led to a revolution in the development of numerical dynamic optimization methods over the past few decades. Among the popular numerical methods for solving optimal control problems, the so-called ``direct orthogonal collocation methods'' and ``direct pseudospectral methods'' have become two of the most universal and well established numerical dynamic optimization methods due to many merits they offer over other competitive methods in the literature; cf. \cite{Fahroo2002,Benson2006,Garg2011b,Elgindy2012d,Elgindy2013b,Elgindy2012c}. Both classes of numerical dynamic optimization methods convert the continuous optimal control problem into a finite dimensional constrained optimization problem based on the elegant spectral and pseudospectral methods, which are known to furnish exponential/spectral convergence rates faster than any polynomial convergence rate when the problem exhibits sufficiently smooth solutions; cf. \cite{Orszag1980,Canuto1988}.

Direct hp-pseudospectral methods were specifically designed to handle optimal control problems with discontinuous or nonsmooth states and controls; cf. \cite{Rao2010,Darby2011,Chai2015,Patterson2015,Patterson2014}. Such methods generally recover the prominent exponential convergence rates of pseudospectral methods by dividing the solution domain into an increasing number of mesh intervals and increasing the degree of the polynomial interpolant within each mesh interval. In particular, a local $p$-refinement is a suitable technique on regions where the solution is smooth, while a local $h$-refinement is preferable on elements where the solution is discontinuous/nonsmooth. To avoid high computational costs, adaptive strategies strive to control the locations of mesh intervals, minimize their number, and answer the question of whether increasing the number of collocation points within each mesh interval is necessary or not to achieve a certain accuracy threshold. 

In the most general formulation of an hp finite element method, the solution over each element is approximated by an arbitrary degree polynomial. The spectral element method uses instead a high-degree piecewise polynomial defined by an appropriate set of interpolation nodes or expansion modes. To achieve the highest interpolation accuracy, the interior interpolation nodes are distributed at positions corresponding to the zeros of certain families of orthogonal polynomials; cf. \cite{Pozrikidis2014}. While direct hp-pseudospectral methods were thoroughly investigated in the past few years, comparable literature for direct adaptive spectral element methods for solving special classes of optimal control problems is rather very few, and to the best of our knowledge, it seems that such methods do not exist for solving more general nonlinear optimal control problems. We acknowledge though the existence of some posteriori error analyses of hp finite element approximations of special forms of convex optimal control problems; cf. \cite{Chen2011,Gong2011}. Posteriori error estimates for the spectral element approximation of a linear quadratic optimal control problem in one dimension was recently presented by \cite{Ye2016}. However, all three papers lacked any adaptive strategies to efficiently implement their numerical schemes. Perhaps, the earliest and sole adaptive spectral element method for solving a special class of optimal control problems described by a quadratic cost functional and linear advection-diffusion state equation was put forward by \cite{Gaudio2011}. In their presented work, an approximate saddle point of the Lagrangian functional is sought by iterating on the Karush-Kuhn-Tucker optimality conditions to seek their satisfaction numerically using a Galerkin spectral element method for the space discretization. The adaptive algorithm relies on a posteriori error estimate of the cost functional, from which the parameters of the spectral element discretization are selected. 

The main purpose of this paper is to derive high-order numerical solutions of nonlinear optimal control problems exhibiting smooth/nonsmooth solutions using a novel direct adaptive Gegenbauer integral spectral element (GISE) method. In particular, the proposed method converts the nonlinear optimal control problem into an integral multiple-phase optimal control problem. The multiple-phases are then connected using state continuity linkage conditions with easily incorporated control continuity linkage conditions when the control functions are assumed continuous. The numerical discretization is carried out using truncated shifted Gegenbauer series expansions and a novel numerical quadrature defined on each mesh interval-- henceforth called the $k$th element shifted optimal barycentric Gegenbauer quadrature (KESOBGQ)-- based on the stable barycentric representation of Lagrange interpolating polynomials. Such a quadrature can produce excellent approximations while significantly reducing the number of operational costs required for the evaluation of the involved integrals. The proposed method is further invigorated by a novel adaptive strategy that uses a multicriterion for locating the mesh intervals where the state and control functions are smooth/nonsmooth based on information derived from the residual of the discrete dynamical system equations, and  the magnitude of the last coefficients in the state and control truncated series. In fact, the idea of using the spectral coefficients of the state trajectories as a measure to verify the convergence of the computed solution was previously presented by \cite{Gong2006}. Nonetheless, in this article, we shall exploit the spectral coefficients instead to check the smoothness of the approximate solutions on the interval of interest. The proposed method generally produces a small/medium-scale nonlinear programming problem that could be easily solved using the current powerful numerical optimization methods. The current paper casts further the light on the judicious choice of the shifted Gegenbauer-Gauss collocation points set to be utilized on each mesh interval during the discretization process of optimal control problems based on numerical simulations.   

The remaining part of the paper is organized as follows: In Section \ref{sec:ps}, we describe the optimal control problem statement under study. In Section \ref{sec:tgsem}, we present our novel GISE method. A novel adaptive strategy is presented in Section \ref{subsec:as1}. Section \ref{subsec:err} is devoted for the error analysis and convergence properties of the KESOBGQ. In Section \ref{sec:ne}, three test examples of nonlinear optimal control problems are included to demonstrate the efficiency and the accuracy of the proposed GISE method followed by some concluding remarks illustrating the advantages of the proposed GISE method in Section \ref{conc}.
\section{Problem statement}
\label{sec:ps}
Consider the nonlinear time-varying dynamical system
\begin{equation}\label{eq:dynsys1}
	\bm{\dot x}(t) = \bm{f}(\bm{x}(t),\bm{u}(t),t),\quad {t_0} \le t \le {t_f},
\end{equation}
where $\bm{x}(t) \in {\mathbb{R}^{{n_x}}}$ and $\bm{u}(t) \in {\mathbb{R}^{{n_u}}}$ are the state and control vector functions for some $n_x, n_u \in \mathbb{Z}^+$, respectively; $\bm{\dot x}(t)$ is the vector of first-order time derivatives of the states; $t_0 \in \mathbb{R}$ is the initial time, $t_f \in \mathbb{R}: t_f > t_0$ is the terminal time. The problem is to find the optimal control $\bm{u}^*(t)$ and the corresponding state trajectory $\bm{x}^*(t), t_0 \le t \le t_f$ satisfying Eq. \eqref{eq:dynsys1} while minimizing the cost functional
\begin{equation}\label{eq:costfn1}
	J = \phi \left( {\bm{x}({t_0}),{t_0},\bm{x}({t_f}),{t_f}} \right) + \int_{{t_0}}^{{t_f}} {\mathcal{L}(\bm{x}(t),\bm{u}(t),t)\,dt} ,
\end{equation}
subject to the mixed state and control path constraints
\begin{equation}\label{ineq:pathcons1}
	{\bm{C}_{\min }} \le \bm{C}(\bm{x}(t),\bm{u}(t),t) \le {\bm{C}_{\max }},
\end{equation}
and the boundary conditions
\begin{equation}\label{eq:bound1}
	\psi \left( {\bm{x}({t_0}),{t_0},\bm{x}({t_f}),{t_f}} \right) = \bm{0},
\end{equation}
where $\phi :{\mathbb{R}^{{n_x}}} \times \mathbb{R} \times {\mathbb{R}^{{n_x}}} \times \mathbb{R} \to \mathbb{R}$ is the terminal cost function, ${\cal L}:{\mathbb{R}^{{n_x}}} \times {\mathbb{R}^{{n_u}}} \times \mathbb{R} \to \mathbb{R}$ is the Lagrangian function, $\bm{f}:{\mathbb{R}^{{n_x}}} \times {\mathbb{R}^{{n_u}}} \times \mathbb{R} \to {\mathbb{R}^{{n_x}}}$ is a nonlinear vector field, $\bm{C}:{\mathbb{R}^{{n_x}}} \times {\mathbb{R}^{{n_u}}} \times \mathbb{R} \to {\mathbb{R}^{{n_C}}}$ is a mixed inequality constraint vector on the state and control functions for some $n_C \in \mathbb{Z}^+$; ${\bm{C}_{\min }}, {\bm{C}_{\max }} \in {\mathbb{R}^{{n_C}}}$ are constant specified vectors, and $\psi :{\mathbb{R}^{{n_x}}} \times \mathbb{R} \times {\mathbb{R}^{{n_x}}} \times \mathbb{R} \to {\mathbb{R}^{{n_\psi }}}$ is a boundary constraint vector for some $n_{\psi} \in \mathbb{Z}^+$. Here it is assumed that $\phi, \cal{L}$, and each system function $f_i$ are nonlinear continuously differentiable functions with respect to $\bm{x}$. It is also assumed that the nonlinear optimal control problem \eqref{eq:dynsys1}--\eqref{eq:bound1} has a unique solution with possibly discontinuous/nonsmooth optimal control. We shall refer to the above optimal control problem in Bolza form by Problem 1.

\section{The GISE method}
\label{sec:tgsem}
Using the affine transformation 
\begin{equation}\label{eq:affine1}
	\tau  = \frac{{2\,t - {t_0} - {t_f}}}{{{t_f} - {t_0}}},
\end{equation}
we could easily rewrite Problem 1 as follows: 
\begin{subequations}\label{pr:2}
\begin{equation}
	{\text{Minimize }}J = \phi \left( {\tilde {\bm{x}}( - 1),{t_0},\tilde {\bm{x}}(1),{t_f}} \right) + \frac{{{t_f} - {t_0}}}{2}\int_{ - 1}^1 {\tilde {\cal{L}}(\tilde {\bm{x}}(\tau ),\tilde {\bm{u}}(\tau ),\tau )\,d\tau }
	\end{equation}
	\text{subject to }	
	\begin{equation}
		\skew{3}{\dot}{\skew{3}{\tilde}{\bm{x}}}(\tau) = \frac{{{t_f} - {t_0}}}{2}\tilde {\bm{f}}(\tilde {\bm{x}}(\tau ),\tilde {\bm{u}}(\tau ),\tau ),\quad \tau  \in [ - 1,1],
	\end{equation}	
	\begin{equation}
		{\bm{C}_{\min }} \le \tilde {\bm{C}}(\tilde {\bm{x}}(\tau ),\tilde {\bm{u}}(\tau ),\tau ) \le {\bm{C}_{\max }},\quad \tau  \in [ - 1,1],
	\end{equation}
	\begin{equation}
		\psi \left( {\tilde {\bm{x}}( - 1),{t_0},\tilde {\bm{x}}(1),{t_f}} \right) = \bm{0},
	\end{equation}
\end{subequations}
where $\tilde \eta (\tau ) = \eta \left( {(({t_f} - {t_0})\,\tau  + {t_0} + {t_f})/2} \right),$ for all $\eta  \in \left\{ {\bm{x},\bm{u},\cal{L},\bm{f},\bm{C}} \right\}$. We refer to the optimal control problem described by Eqs. \eqref{pr:2} by Problem 2.

One of the primary advantages of spectral element methods is the ability to resolve complex geometries and problems exhibiting discontinuous/nonsmooth solutions with high-order accuracies through the decomposition of the solution interval into small mesh intervals or elements ``$h$-refinement,'' and approximating the restricted solution function on each mesh interval with high-order truncated spectral expansion series ``$p$-refinement.'' Considering the solution interval $[-1,1]$, we can partition it into $K$ mesh intervals $\bm{\Omega}_k, k \in \mathbb{K} = \{1, \ldots, K\}$ using $K+1$ mesh points $\tau_k, k = 0, \ldots, K$ distributed along the interval $[-1,1]$:
\[[ - 1,1] = \bigcup\limits_{k = 1}^K {{\mkern 1mu} {\bm{\Omega} _k}} ,\quad {\bm{\Omega} _k} = [{\tau _{k - 1}},{\tau _k}],\quad - 1 = {\tau _0} < {\tau _1} <  \ldots  < {\tau _K} = 1.\]
For simplicity of notation, we explicitly write $\tau ^{(k)}$ to denote the restricted variable $\tau$ whose values are confined to ${\bm{\Omega} _k}$; i.e.
$\tau^{(k)} = \tau: \tau_{k-1} \le \tau \le \tau_k$. Moreover, we denote the state and control vector functions in the $k$th element by ${{\tilde {\bm{x}}}^{(k)}}\left( {{\tau ^{(k)}}} \right)$ and ${{\tilde {\bm{u}}}^{(k)}}\left( {{\tau ^{(k)}}} \right)$, respectively. Based on this initial setting, we can put Problem 2 into its multiple-interval form as follows:
\begin{subequations}\label{pr:3}
\begin{equation}\label{eq:costfn3}
	{\text{Minimize }}J = \phi \left( {\tilde {\bm{x}}^{(1)}( - 1),{t_0},\tilde {\bm{x}}^{(K)}(1),{t_f}} \right) + \frac{{{t_f} - {t_0}}}{2} \sum\limits_{k = 1}^K {\int_{ \tau_{k-1}}^{\tau_k} {\tilde {\cal{L}}\left(\tilde {\bm{x}}^{(k)}\left(\tau^{(k)} \right),\tilde {\bm{u}}^{(k)}\left(\tau^{(k)} \right),\tau^{(k)} \right)\,d\tau^{(k)} }}
\end{equation}	
	\text{subject to }
\begin{equation}\label{eq:dynsys3}
	\skew{3}{\dot}{\skew{3}{\tilde}{\bm{x}}}^{(k)}\left(\tau^{(k)}\right) = \frac{{{t_f} - {t_0}}}{2}\tilde {\bm{f}}\left(\tilde {\bm{x}}^{(k)}\left(\tau^{(k)} \right),\tilde {\bm{u}}^{(k)}\left(\tau^{(k)} \right),\tau^{(k)} \right),
\end{equation}
	\begin{equation}\label{eq:mixpathcons3}
		{\bm{C}_{\min }} \le \tilde {\bm{C}}\left(\tilde {\bm{x}}^{(k)}\left(\tau^{(k)} \right),\tilde {\bm{u}}^{(k)}\left(\tau^{(k)} \right),\tau^{(k)} \right) \le {\bm{C}_{\max }},
	\end{equation}
\begin{equation}\label{eq:tercons3}
	\psi \left( {\tilde {\bm{x}}^{(1)}( - 1),{t_0},\tilde {\bm{x}}^{(K)}(1),{t_f}} \right) = \bm{0}.
\end{equation}
\end{subequations}
To take advantage of the well-conditioning of numerical integration operators, we further rewrite Eq. \eqref{eq:dynsys3} in its integral formulation so that
\begin{equation}\label{eq:dynsys4}
	{{\tilde {\bm{x}}}^{(k)}}\left( {{\tau ^{(k)}}} \right) = {{\tilde {\bm{x}}}^{(k)}}\left( {{\tau _{k - 1}}} \right) + \frac{{{t_f} - {t_0}}}{2}\int_{{\tau _{k - 1}}}^{{\tau ^{(k)}}} {\tilde {\bm{f}}\left( {{{\tilde {\bm{x}}}^{(k)}}\left( {{\tau ^{(k)}}} \right),{{\tilde {\bm{u}}}^{(k)}}\left( {{\tau ^{(k)}}} \right),{\tau ^{(k)}}} \right)\,d{\tau ^{(k)}}},\quad k \in \mathbb{K}.
\end{equation}
To impose the states continuity conditions, the following conditions must be fulfilled at the interface of any two consecutive mesh intervals:
\begin{equation}\label{eq:statecont2}
	{{\tilde {\bm{x}}}^{(k - 1)}}\left( {{\tau _{k - 1}}} \right) = {{\tilde {\bm{x}}}^{(k)}}\left( {{\tau _{k - 1}}} \right),\quad k = 2, \ldots ,K.
\end{equation}
If the control vector function is assumed to be continuous, then we further add the following constraints:
\begin{equation}
		{{\tilde {\bm{u}}}^{(k - 1)}}\left( {{\tau _{k - 1}}} \right) = {{\tilde {\bm{u}}}^{(k)}}\left( {{\tau _{k - 1}}} \right),\quad k = 2, \ldots ,K.\label{eq:controlcont2}
\end{equation}
We refer to the optimal control problem \eqref{eq:costfn3}, \eqref{eq:mixpathcons3}, \eqref{eq:tercons3}, \eqref{eq:dynsys4} provided with 
Conditions \eqref{eq:statecont2} and \eqref{eq:controlcont2} by Problem 3.

Let $\mathbb{Z}^+_0 = \mathbb{Z}^+ \cup \{0\}, \mathbb{R}^+_0 = \mathbb{R}^+ \cup \{0\}$, and $\hat G_{j,k}^{(\alpha )}\left( {{\tau ^{(k)}}} \right) = G_{j}^{(\alpha )}\left( {\left( {2\,{\tau ^{(k)}} - {\tau _{k - 1}} - {\tau _k}} \right)/({\tau _k} - {\tau _{k - 1}})} \right)$, for all $j \in \mathbb{Z}^+_0$ be the $j$th-degree shifted Gegenbauer polynomial defined on the mesh interval $\bm{\Omega}_k$, for all $k \in \mathbb{K}$-- henceforth referred to by the $j$th-degree $k$th element shifted Gegenbauer polynomial, where $G_{j}^{(\alpha )}(\tau)$ is the classical $j$th-degree Gegenbauer polynomial associated with the real parameter $\alpha > -1/2$; cf. \cite{Elgindy2013,Elgindy2016a}. Moreover, let $\mathbb{S}_{N_k}^{(\alpha)} = \left\{\hat \tau_{N_k,l}^{(k),\alpha}, l = 0, \ldots, N_k\right\}$ denote the set of the zeroes\footnote{The $k$th element shifted Gegenbauer-Gauss (KESGG) nodes in $\bm{\Omega}_k$, for all $k \in \mathbb{K}$.} of the $(N_k+1)$th-degree $k$th element shifted Gegenbauer polynomial, $\hat G_{{N_k} + 1,k}^{(\alpha )}\left( {{\tau ^{(k)}}} \right)$, for some $N_k \in \mathbb{Z}^+, k \in \mathbb{K}$, and set $\hat \tau _{{N_k},{N_k} + 1}^{(k),\alpha } = {\tau _k}$, for all $ k$. The $k$th element shifted Gegenbauer polynomials $\hat G_{n,k}^{(\alpha )}\left( {{\tau ^{(k)}}} \right), n \in \mathbb{Z}^+_0$ form a complete $L_{w_k^{(\alpha)}}^2\left(\bm{\Omega}_k\right)$-orthogonal system with respect to the weight function
\begin{equation}
	w_k^{(\alpha )}\left({\tau ^{(k)}}\right) = {\left( {{\tau _k} - {\tau ^{(k)}}} \right)^{\alpha  - 1/2}}{\left( {{\tau ^{(k)}} - {\tau _{k - 1}}} \right)^{\alpha  - 1/2}},
\end{equation}
 and their orthogonality relation is defined by the following weighted inner product:
\begin{equation}
	{\left( {\hat G_{m,k}^{(\alpha )},\hat G_{n,k}^{(\alpha )}} \right)_{w_k^{(\alpha )}}} = \int_{{\tau _{k - 1}}}^{{\tau _k}} {\hat G_{m,k}^{(\alpha )}\left( {{\tau ^{(k)}}} \right)\hat G_{n,k}^{(\alpha )}\left( {{\tau ^{(k)}}} \right){\mkern 1mu} w_k^{(\alpha )}\left( {{\tau ^{(k)}}} \right)d{\tau ^{(k)}}}  = \left\| {\hat G_{n,k}^{(\alpha )}} \right\|_{w_k^{(\alpha )}}^2{\delta _{m,n}} = \lambda _{k,n}^{(\alpha )}{\delta _{m,n}}\quad \forall m, n \in \mathbb{Z}^+_0,
\end{equation}
where $\delta _{m,n}$ is the Kronecker delta function, 
\begin{equation}
\lambda _{k,n}^{(\alpha )} = {\left( {\frac{{{\tau _k} - {\tau _{k - 1}}}}{2}} \right)^{2\alpha }}\lambda _n^{(\alpha )},
\end{equation}
is the normalization factor, and $\lambda _n^{(\alpha )}$ is as defined by \cite[Eq. (2.6)]{Elgindy2016}. For $\alpha = 0$ and $0.5$, we recover the shifted Chebyshev polynomials of the first kind and the shifted Legendre polynomials, respectively, on each mesh interval $\bm{\Omega}_k, k \in \mathbb{K}$. Let $L_{x,k}, L_{u,k} \in \mathbb{Z}_0^+$, for all $ k \in \mathbb{K}$,  
\[\hat{\bm{ G}}_{L,k}^{(\alpha )}\left( {{\tau ^{(k)}}} \right) = \left[ {\hat{\bm{ G}}_{0,k}^{(\alpha )}\left( {{\tau ^{(k)}}} \right),\hat{\bm{ G}}_{1,k}^{(\alpha )}\left( {{\tau ^{(k)}}} \right), \ldots ,\hat{\bm{ G}}_{L,k}^{(\alpha )}\left( {{\tau ^{(k)}}} \right)} \right]\quad \forall L \in \mathbb{Z}_0^ + ,\]
and denote the identity matrix of order $n$ by $\mathbf{I}_n$, for all $ n \in \mathbb{Z}^+$. Moreover, define ${{\bm{a}}^{(k)}} = \text{vec}\left[ {{\bm{a}}_1^{(k)},{\bm{a}}_2^{(k)}, \ldots ,{\bm{a}}_{{n_x}}^{(k)}} \right]$ and ${{\bm{b}}^{(k)}} = \text{vec}\left[ {{\bm{b}}_1^{(k)},{\bm{b}}_2^{(k)}, \ldots ,{\bm{b}}_{{n_u}}^{(k)}} \right]$ as the spectral coefficient vectors obtained through collocating the state and control vectors 
at the augmented KESGG nodes $\hat \tau_{N_k,i}^{(k),\alpha} \in \mathbb{S}_{N_k}^{(\alpha)}, i = 0, \ldots, N_k+1$, respectively, where ${\bm{a}}_r^{(k)} = {\left[ {a_{r,0}^{(k)},a_{r,1}^{(k)}, \ldots ,a_{r,{L_{x,k}}}^{(k)}} \right]^T},{\bm{b}}_s^{(k)} = {\left[ {b_{s,0}^{(k)},b_{s,1}^{(k)}, \ldots ,b_{s,{L_{u,k}}}^{(k)}} \right]^T}$, for all $ k \in \mathbb{K}, r = 1, \ldots, n_x; s = 1, \ldots, n_u$, and ``vec'' denotes the vectorization of a matrix. As promised by the Sturm-Liouville theorem, we can represent any square integrable function as an infinite series in the shifted Gegenbauer polynomials; therefore, we can approximate the state and control vector functions as follows:
\begin{subequations}\label{eq:statesandcontrolsapprox1}
\begin{align}
	{{\tilde {\bm{x}}}^{(k)}}\left( {{\tau ^{(k)}}} \right) \approx {{\hat {\bm{x}}}^{(k)}}\left( {{\tau ^{(k)}}} \right) &= \left( {{\mathbf{I}_{{n_x}}} \otimes \hat{\bm{ G}}_{{L_{x,k}},k}^{(\alpha )}\left( {{\tau ^{(k)}}} \right)} \right){{\bm{a}}^{(k)}},\\
	{{\tilde {\bm{u}}}^{(k)}}\left( {{\tau ^{(k)}}} \right) \approx {{\hat {\bm{u}}}^{(k)}}\left( {{\tau ^{(k)}}} \right) &= \left( {{\mathbf{I}_{{n_u}}} \otimes \hat{\bm{ G}}_{{L_{u,k}},k}^{(\alpha )}\left( {{\tau ^{(k)}}} \right)} \right){{\bm{b}}^{(k)}},
\end{align}
\end{subequations}
where ``$\otimes$'' denotes the Kronecker product. Let $M_k \in \mathbb{Z}^+ $, for all $ k \in \mathbb{K}, {\mathbf{P}}_{OB}^{(1)} \in \mathbb{R}^{(N_k+2) \times (M_k+1)}$ denote the first-order optimal barycentric Gegenbauer integration matrix (OBGIM); cf. \cite{Elgindy2016a}. With a simple mathematical manipulation, we can easily show that the first-order $k$th element shifted optimal barycentric Gegenbauer integration matrix (KESOBGIM), ${}_k{\mathbf{P}}_{OB}^{(1)}$, on $\bm{\Omega}_k$ is related to the OBGIM by the following useful relation:
\begin{equation}
	{}_k{\mathbf{P}}_{OB}^{(1)} = \frac{1}{2}({\tau _k} - {\tau _{k - 1}}){\mathbf{P}}_{OB}^{(1)}\quad \forall k \in \mathbb{K}.
\end{equation}
We shall refer to the set $\mathbb{T}_{N_k,M_k} = \left\{ \hat z_{M_k,i,j}^{(k),\alpha _i^{(k),*}} ,i = 0, \ldots ,N_k+1; j = 0,  \ldots ,M_k\right\}$ by the set of adjoint KESGG points on $\bm{\Omega}_k$, for all $ k \in \mathbb{K}$ such that $\alpha _i^{(k),*}, i = 0, \ldots ,N_k+1$ are the associated optimal Gegenbauer parameters; cf. \cite{Elgindy2016}. Denote the $i$th row of the KESOBGIM, $\left[ {{}_kp_{OB,i,0}^{(1)}, \ldots ,{}_kp_{OB,i,{M_k}}^{(1)}} \right]$, by ${}_k{\mathbf{P}}_{OB,i}^{(1)}$, for all $ i = 0,\ldots, N_k+1$. The sought discrete cost function $J^{(\alpha)}_{N, L_x, L_u}$ can be written as
\begin{equation}\label{eq:disccost1}
	{J^{(\alpha)}_{N,{L_x},{L_u}}} = \phi \left( {\left( {{{\mathbf{I}}_{{n_x}}} \otimes {{\left( {{\bm{1}}_{{L_{x,1}} + 1}^{( - )}} \right)}^T}} \right){{\bm{a}}^{(1)}},{t_0},\left( {{{\mathbf{I}}_{{n_x}}} \otimes {\bm{1}}_{{L_{x,K}} + 1}^T} \right){{\bm{a}}^{(K)}},{t_f}} \right) + \frac{{{t_f} - {t_0}}}{2}\sum\limits_{k = 1}^K {{}_k{\mathbf{P}}_{OB,{N_k} + 1}^{(1)}\,{\bm{\chi} ^{(k)}}} ,
\end{equation}
where ${{\bm{1}}_L} \in {\mathbb{R}^L}$ is the all ones vector, ${\bm{1}}_L^{( - )} \in {\mathbb{R}^L}:{\left( {{\bm{1}}_L^{( - )}} \right)_i} = {( - 1)^{i - 1}},i = 1, \ldots ,L,$ is the all alternating ones vector for all $L \in \mathbb{Z}^+$, and 
\begin{align}
	{\bm{\chi} ^{(k)}} = \left[ \tilde {\cal{L}}\left( {\left( {{{\mathbf{I}}_{{n_x}}} \otimes \hat {\bm{G}}_{{L_{x,k}},k}^{(\alpha )}\left( {\hat z_{{M_k},{N_k} + 1,0}^{(k),\alpha _{{N_k} + 1}^{(k),*}}} \right)} \right){{\bm{a}}^{(k)}},\left( {{{\mathbf{I}}_{{n_u}}} \otimes \hat {\bm{G}}_{{L_{u,k}},k}^{(\alpha )}\left( {\hat z_{{M_k},{N_k} + 1,0}^{(k),\alpha _{{N_k} + 1}^{(k),*}}} \right)} \right){{\bm{b}}^{(k)}},\hat z_{{M_k},{N_k} + 1,0}^{(k),\alpha _{{N_k} + 1}^{(k),*}}} \right), \ldots ,\right.\nonumber\\
	\left.\tilde {\cal{L}}\left( {\left( {{{\mathbf{I}}_{{n_x}}} \otimes \hat {\bm{G}}_{{L_{x,k}},k}^{(\alpha )}\left( {\hat z_{{M_k},{N_k} + 1,{M_k}}^{(k),\alpha _{{N_k} + 1}^{(k),*}}} \right)} \right){{\bm{a}}^{(k)}},\left( {{{\mathbf{I}}_{{n_u}}} \otimes \hat {\bm{G}}_{{L_{u,k}},k}^{(\alpha )}\left( {\hat z_{{M_k},{N_k} + 1,{M_k}}^{(k),\alpha _{{N_k} + 1}^{(k),*}}} \right)} \right){{\bm{b}}^{(k)}},\hat z_{{M_k},{N_k} + 1,{M_k}}^{(k),\alpha _{{N_k} + 1}^{(k),*}}} \right) \right]^T.
\end{align}
To account for the state continuity conditions \eqref{eq:statecont2}, the discrete integral dynamical system equations on the elemental domains (or simply elements) can be approximated by
\begin{subequations}\label{eq:discdynsys1}
\begin{equation}
\left( {{{\mathbf{I}}_{{n_x}}} \otimes \left( {\hat{\bm{ G}}_{{L_{x,1}},1}^{(\alpha )}\left( {\hat \tau _{{N_1},i}^{(1),\alpha }} \right) - {{\left( {{\bm{1}}_{{L_{x,1}} + 1}^{( - )}} \right)}^T}} \right)} \right){{\bm{a}}^{(1)}} - \frac{{{t_f} - {t_0}}}{2}\left( {{{\mathbf{I}}_{{n_x}}} \otimes {}_1{\mathbf{P}}_{OB,i}^{(1)}} \right)\hat{\bm{F}}_i^{(1)} = \bm{0},\quad i = 0, \ldots ,{N_1} + 1,
\end{equation}
\begin{align}
\left( {{{\mathbf{I}}_{{n_x}}} \otimes \hat{\bm{ G}}_{{L_{x,k}},k}^{(\alpha )}\left( {\hat \tau _{{N_k},i}^{(k),\alpha }} \right)} \right){{\bm{a}}^{(k)}} - \left( {{{\mathbf{I}}_{{n_x}}} \otimes {\bm{1}}_{{L_{x,k - 1}} + 1}^T} \right){{\bm{a}}^{(k - 1)}} - \frac{{{t_f} - {t_0}}}{2}\left( {{{\mathbf{I}}_{{n_x}}} \otimes {}_k{\mathbf{P}}_{OB,i}^{(1)}} \right)\hat {\bm{F}}_i^{(k)} = \bm{0},\quad &i = 0, \ldots ,{N_k} + 1;\nonumber\\ &k = 2, \ldots ,K,
\end{align}
\end{subequations}
where 
\begin{align}
	\hat {\bm{F}}_i^{(k)} &= \left[ {{\tilde f}_1}\left( {\left( {{{\mathbf{I}}_{{n_x}}} \otimes \hat{\bm{ G}}_{{L_{x,k}},k}^{(\alpha )}\left( {\hat z_{{M_k},i,0}^{(k),\alpha _i^{(k),*}}} \right)} \right){{\bm{a}}^{(k)}},\left( {{{\mathbf{I}}_{{n_u}}} \otimes \hat{\bm{ G}}_{{L_{u,k}},k}^{(\alpha )}\left( {\hat z_{{M_k},i,0}^{(k),\alpha _i^{(k),*}}} \right)} \right){{\bm{b}}^{(k)}},\hat z_{{M_k},i,0}^{(k),\alpha _i^{(k),*}}} \right), \ldots ,\right.\nonumber\\
	&\left.{{\tilde f}_1}\left( {\left( {{{\mathbf{I}}_{{n_x}}} \otimes \hat{\bm{ G}}_{{L_{x,k}},k}^{(\alpha )}\left( {\hat z_{{M_k},i,{M_k}}^{(k),\alpha _i^{(k),*}}} \right)} \right){{\bm{a}}^{(k)}},\left( {{{\mathbf{I}}_{{n_u}}} \otimes \hat{\bm{ G}}_{{L_{u,k}},k}^{(\alpha )}\left( {\hat z_{{M_k},i,{M_k}}^{(k),\alpha _i^{(k),*}}} \right)} \right){{\bm{b}}^{(k)}},\hat z_{{M_k},i,{M_k}}^{(k),\alpha _i^{(k),*}}} \right), \ldots,\right.\nonumber\\
	&\left.{{\tilde f}_{n_x}}\left( {\left( {{{\mathbf{I}}_{{n_x}}} \otimes \hat{\bm{ G}}_{{L_{x,k}},k}^{(\alpha )}\left( {\hat z_{{M_k},i,0}^{(k),\alpha _i^{(k),*}}} \right)} \right){{\bm{a}}^{(k)}},\left( {{{\mathbf{I}}_{{n_u}}} \otimes \hat{\bm{ G}}_{{L_{u,k}},k}^{(\alpha )}\left( {\hat z_{{M_k},i,0}^{(k),\alpha _i^{(k),*}}} \right)} \right){{\bm{b}}^{(k)}},\hat z_{{M_k},i,0}^{(k),\alpha _i^{(k),*}}} \right), \ldots ,\right.\nonumber\\
	&\left.{{\tilde f}_{n_x}}\left( {\left( {{{\mathbf{I}}_{{n_x}}} \otimes \hat{\bm{ G}}_{{L_{x,k}},k}^{(\alpha )}\left( {\hat z_{{M_k},i,{M_k}}^{(k),\alpha _i^{(k),*}}} \right)} \right){{\bm{a}}^{(k)}},\left( {{{\mathbf{I}}_{{n_u}}} \otimes \hat{\bm{ G}}_{{L_{u,k}},k}^{(\alpha )}\left( {\hat z_{{M_k},i,{M_k}}^{(k),\alpha _i^{(k),*}}} \right)} \right){{\bm{b}}^{(k)}},\hat z_{{M_k},i,{M_k}}^{(k),\alpha _i^{(k),*}}} \right)\right]^T,\nonumber\\
	&\quad i = 0,\ldots ,{N_k} + 1; k \in \mathbb{K}.
\end{align}
Furthermore, the discrete path and boundary constraints are given by
\begin{equation}\label{eq:discpath1}
	{{\bm{C}}_{\min }} \le \tilde {\bm{C}}\left( {\left( {{{\mathbf{I}}_{{n_x}}} \otimes \hat{\bm{G}}_{{L_{x,k}},k}^{(\alpha )}\left( {\hat \tau _{{N_k},i}^{(k),\alpha }} \right)} \right){{\bm{a}}^{(k)}},\left( {{{\mathbf{I}}_{{n_u}}} \otimes \hat {\bm{G}}_{{L_{u,k}},k}^{(\alpha )}\left( {\hat \tau _{{N_k},i}^{(k),\alpha }} \right)} \right){{\bm{b}}^{(k)}},\hat \tau _{{N_k},i}^{(k),\alpha }} \right) \le {{\bm{C}}_{\max }},\quad i = 0, \ldots ,{N_k} + 1; k \in \mathbb{K},
\end{equation}
\begin{equation}\label{eq:discbound1}
	\psi \left( {\left( {{{\mathbf{I}}_{{n_x}}} \otimes {{\left( {{\bm{1}}_{{L_{x,1}} + 1}^{( - )}} \right)}^T}} \right){{\bm{a}}^{(1)}},{t_0},\left( {{{\mathbf{I}}_{{n_x}}} \otimes {\bm{1}}_{{L_{x,K}} + 1}^T} \right){{\bm{a}}^{(K)}},{t_f}} \right) = {\bm{0}}.
\end{equation}
The discrete control continuity constraints \eqref{eq:controlcont2} are imposed as follows:
\begin{equation}\label{eq:disccont1}
	\left( {{{\mathbf{I}}_{{n_u}}} \otimes {{\left( {{\bm{1}}_{{L_{u,k}} + 1}^{( - )}} \right)}^T}} \right){{\bm{b}}^{(k)}} - \left( {{{\mathbf{I}}_{{n_u}}} \otimes {\bm{1}}_{{L_{u,k - 1}} + 1}^T} \right){{\bm{b}}^{(k - 1)}} = {\bm{0}},\quad k = 2,\ldots, K.
\end{equation}
Hence, the optimal control problem has been reduced to a nonlinear programming problem in which we seek the minimization of the objective function ${J^{(\alpha)}_{N,{L_x},{L_u}}}$ defined by Eq. \eqref{eq:disccost1} subject to the generally nonlinear constraints \eqref{eq:discdynsys1}, \eqref{eq:discpath1}, \eqref{eq:discbound1}, and the linear constraints \eqref{eq:disccont1}. 
\begin{rem}
The present GISE method adopts both collocation and interpolation techniques to obtain the sought approximations. In particular, the spectral coefficient vectors are determined through collocation at the augmented KESGG nodes on each element $\bm{\Omega}_k$, while the KESOBGIMs are constructed through interpolation at the KESGG nodes.
\end{rem}
\subsection{Adaptive Strategy}
\label{subsec:as1}
In this section, we present a multicriterion for locating the elements where the state and control functions are smooth/nonsmooth by monitoring: (i) the maximum residual of the discrete dynamical system equations; i.e., checking whether the state and control variables at the midpoints of each segment joining two consecutive discretization points on the same element meet the restrictions of the dynamical system equations, (ii) the magnitude of the last coefficients in the state and control truncated series.

To illustrate the proposed adaptive technique, let us begin by defining the element midpoints vector $\check{\bm{\tau }}_{{N_k}}^{(k),\alpha } = {\left[ {\check\tau _{{N_k},0}^{(k),\alpha }, \ldots ,\check\tau _{{N_k},{N_k}}^{(k),\alpha }} \right]^T}: \check\tau _{{N_k},i}^{(k),\alpha } = \frac{1}{2}\left( {\hat \tau _{{N_k},i}^{(k),\alpha } + \hat \tau _{{N_k},i + 1}^{(k),\alpha }} \right),\; i = 0, \ldots ,{N_k}; k \in \mathbb{K}.$ Let $\bar M_k \in \mathbb{Z}^+ $, for all $ k \in \mathbb{K}$, and denote the KESOBGIM constructed using the integration nodes vector $\check{\bm{\tau }}_{{N_k}}^{(k),\alpha }$ by ${}_k{\check{\mathbf{P}}}_{OB}^{(1)} $, for all $ k \in \mathbb{K}$; cf. \cite{Elgindy2016a}. On each element $\bm{\Omega}_k$, define the absolute residual matrix ${{\mathbf{R}}^{(k)}} = \left[ {{\bm{R}}_0^{(k)}; \ldots ;{\bm{R}}_{{N_k}}^{(k)}} \right]:$
\begin{equation}
	{\bm{R}}_i^{(k)} = \left| {{{\left( {\left( {{{\mathbf{I}}_{{n_x}}} \otimes \left( {\hat{\bm{ G}}_{{L_{x,k}},k}^{(\alpha )}\left( {\check \tau _{{N_k},i}^{(k),\alpha }} \right) - {{\left( {{\bm{1}}_{{L_{x,k}} + 1}^{( - )}} \right)}^T}} \right)} \right){{\bm{a}}^{(k)}} - \frac{{{t_f} - {t_0}}}{2}\left( {{{\mathbf{I}}_{{n_x}}}{ \otimes _k}\check{\mathbf{P}}_{OB,i}^{(1)}} \right) \check{\bm{F}}_i^{(k)}} \right)}^T}} \right|,\quad i = 0, \ldots ,{N_k},
\end{equation}
where ``$[\cdot; \cdot]$'' is the vertical matrix concatenation along columns, and
\begin{align}
	\check {\bm{F}}_i^{(k)} &= \left[ {{\tilde f}_1}\left( {\left( {{{\mathbf{I}}_{{n_x}}} \otimes \hat{\bm{ G}}_{{L_{x,k}},k}^{(\alpha )}\left( {\check z_{{\bar{M_k}},i,0}^{(k),\alpha _i^{(k),*}}} \right)} \right){{\bm{a}}^{(k)}},\left( {{{\mathbf{I}}_{{n_u}}} \otimes \hat{\bm{ G}}_{{L_{u,k}},k}^{(\alpha )}\left( {\check z_{{\bar{M_k}},i,0}^{(k),\alpha _i^{(k),*}}} \right)} \right){{\bm{b}}^{(k)}},\check z_{{\bar{M_k}},i,0}^{(k),\alpha _i^{(k),*}}} \right), \ldots ,\right.\nonumber\\
	&\left.{{\tilde f}_1}\left( {\left( {{{\mathbf{I}}_{{n_x}}} \otimes \hat{\bm{ G}}_{{L_{x,k}},k}^{(\alpha )}\left( {\check z_{{\bar{M_k}},i,{\bar{M_k}}}^{(k),\alpha _i^{(k),*}}} \right)} \right){{\bm{a}}^{(k)}},\left( {{{\mathbf{I}}_{{n_u}}} \otimes \hat{\bm{ G}}_{{L_{u,k}},k}^{(\alpha )}\left( {\check z_{{\bar{M_k}},i,{\bar{M_k}}}^{(k),\alpha _i^{(k),*}}} \right)} \right){{\bm{b}}^{(k)}},\check z_{{\bar{M_k}},i,{\bar{M_k}}}^{(k),\alpha _i^{(k),*}}} \right), \ldots,\right.\nonumber\\
	&\left.{{\tilde f}_{n_x}}\left( {\left( {{{\mathbf{I}}_{{n_x}}} \otimes \hat{\bm{ G}}_{{L_{x,k}},k}^{(\alpha )}\left( {\check z_{{\bar{M_k}},i,0}^{(k),\alpha _i^{(k),*}}} \right)} \right){{\bm{a}}^{(k)}},\left( {{{\mathbf{I}}_{{n_u}}} \otimes \hat{\bm{ G}}_{{L_{u,k}},k}^{(\alpha )}\left( {\check z_{{\bar{M_k}},i,0}^{(k),\alpha _i^{(k),*}}} \right)} \right){{\bm{b}}^{(k)}},\check z_{{\bar{M_k}},i,0}^{(k),\alpha _i^{(k),*}}} \right), \ldots ,\right.\nonumber\\
	&\left.{{\tilde f}_{n_x}}\left( {\left( {{{\mathbf{I}}_{{n_x}}} \otimes \hat{\bm{ G}}_{{L_{x,k}},k}^{(\alpha )}\left( {\check z_{{\bar{M_k}},i,{\bar{M_k}}}^{(k),\alpha _i^{(k),*}}} \right)} \right){{\bm{a}}^{(k)}},\left( {{{\mathbf{I}}_{{n_u}}} \otimes \hat{\bm{ G}}_{{L_{u,k}},k}^{(\alpha )}\left( {\check z_{{\bar{M_k}},i,{\bar{M_k}}}^{(k),\alpha _i^{(k),*}}} \right)} \right){{\bm{b}}^{(k)}},\check z_{{\bar{M_k}},i,{\bar{M_k}}}^{(k),\alpha _i^{(k),*}}} \right)\right]^T,\quad i = 0, \ldots ,{N_k}.
\end{align}
Now, let ${i_{\max }} = i,{j_{\max }} = j:\left({{\mathbf{R}}^{(k)}}\right)_{{i_{\max }},{j_{\max }}} = \mathop {\max }\limits_{i,j} \left({{\mathbf{R}}^{(k)}}\right)_{i,j}$. If $\left({{\mathbf{R}}^{(k)}}\right)_{{i_{\max }},{j_{\max }}}$ falls below $\epsilon_{\mathbf{R}}$, a user-specified threshold for the size of the residual error, then the obtained state and control spectral coefficients could be optimal. As a further measure of the quality of the determined coefficients, we check the magnitude of the last coefficients in the state and control truncated series. In fact, for sufficiently smooth functions, the $i$th coefficient of the spectral expansion generally decay faster than any inverse power of $i$, for all $ i \in \mathbb{Z}^+$; cf. \cite{Elgindy2013b}. This fundamental property of spectral methods can be exploited and combined with the residual error measurement to form a strong tool used to decide whether to accept the obtained approximations, or not. We shall refer to the condition $\left({{\mathbf{R}}^{(k)}}\right)_{{i_{\max }},{j_{\max }}} < \epsilon_{\mathbf{R}}$ by Condition $\cal{A}$ and refer to the inequalities $\left| {a_{r,{L_{x,k}}}^{(k)}} \right|,\left| {b_{s,{L_{u,k}}}^{(k)}} \right| < \epsilon_{\text{coeff}}$, for all $ r,s$ by Condition $\cal{B}$, where $\epsilon_{\text{coeff}}$ is a user-specified threshold for the size of the last spectral coefficients .

If both Conditions $\cal{A}$ and $\cal{B}$ are satisfied, the obtained approximations are considered acceptable. If not, then we need to decide whether to divide the current mesh interval $\bm{\Omega}_k$, or increase the number of collocation points and spectral coefficients. To this end, we slightly follow the approach presented by \cite{Darby2011}. In particular, we determine the column vector, ${{\bm{r}}^{(k)}} = \left[ {{\left( {{{\mathbf{R}}^{(k)}}} \right)}_{0,{j_{\max }}}}, \ldots ,{{\left( {{{\mathbf{R}}^{(k)}}} \right)}_{{N_k},{j_{\max }}}} \right]^T$ from the largest element of ${\mathbf{R}}^{(k)}$, and calculate the arithmetic mean, ${{\bar r}^{(k)}}$, of the elements of ${{\bm{r}}^{(k)}}$ by the following formula:
\[{{\bar r}^{(k)}} = \frac{{\sum\nolimits_{i = 0}^{{N_k}} {{{\left( {{{\mathbf{R}}^{(k)}}} \right)}_{i,{j_{\max }}}}} }}{{{N_k} + 1}}.\]
Finally, we find the residual vector ${{\bm{\beta }}^{(k)}}$ via calculating ${{\bm{r}}^{(k)}}/{{\bar r}^{(k)}}$. Now, let $\rho > 1$ be a user-specified threshold for the size of the elements $\beta_i^{(k)}, i = 0, \ldots, N_k$ of the vector ${{\bm{\beta }}^{(k)}}$, and define a discrete local maximum (peak) $\beta_{i,dlm}^{(k)}$ of ${{\bm{\beta }}^{(k)}}$ by the data sample that is larger than its two neighboring samples; i.e., the value $\beta_i: \beta_{i-1} < \beta_{i} > \beta_{i+1}, i = 1, \ldots, N_k-1$. Let $\bm{\beta}_{dlm}^{(k)}$ be the row vector of the local maxima of ${{\bm{\beta }}^{(k)}}$. We have the following three cases:
\begin{description}
	\item[$\mathbf{(i)}$] If $\beta_0 > \beta_1 \wedge\; \beta_{N_k} > \beta_{N_k-1}$, then set $\bm{\beta}_{dlm}^{(k)} := \left[\beta_0, {\bm{\beta}_{dlm}^{(k)}}, \beta_{N_k}\right]$.
	\item[$\mathbf{(ii)}$] If $\beta_0 > \beta_1 \wedge\; \beta_{N_k} \le \beta_{N_k-1}$, then set $\bm{\beta}_{dlm}^{(k)} := \left[\beta_0, {\bm{\beta}_{dlm}^{(k)}}\right]$.
	\item[$\mathbf{(iii)}$] If $\beta_0 \le \beta_1 \wedge\; \beta_{N_k} > \beta_{N_k-1}$, then set $\bm{\beta}_{dlm}^{(k)} := \left[{\bm{\beta}_{dlm}^{(k)}}, \beta_{N_k}\right]$.
\end{description}
If the error is nonuniform, we break the domain $\bm{\Omega}_k$ at the element midpoints $\check \tau _{{N_k},j}^{(k),\alpha }: \beta_{j,dlm}^{(k)} > \rho $, for all $ j$. Otherwise, the error is considered uniform, so we increase the number of collocation points and spectral coefficients by some constant values as long as the degree of the Gegenbauer polynomial interpolant remains below a maximum allowable degree. In particular, we choose some positive integer numbers $\bar N_k, \bar L_{x,k}, \bar L_{u,k}, N_{k,\max}, L_{x,k,\max}, L_{u,k,\max}$, and update the values of $N_k, L_{x,k}$, and $L_{u,k}: N_k:= N_k + \bar N_k \le N_{k,\max}, L_{x,k}:= L_{x,k} + \bar L_{x,k} \le L_{x,k,\max}, L_{u,k}:= L_{u,k} + \bar L_{u,k} \le L_{u,k,\max}$. In the former case, the interval partitioning is only allowed for a maximum number of divisions $k_{\max} \in \mathbb{Z}^+$. Moreover, to prevent the division of a relatively small domain, we introduce the ``edge spacing'' parameter $\epsilon_{ES}$, so that further domain partitioning is forbidden if $\left|\bm{\Omega}_k \right| < \epsilon_{ES}$, for all $ k \in \mathbb{K}$, where $\left|\bm{\Omega}_k \right|$ denotes the length of the interval $\bm{\Omega}_k$, for all $ k \in \mathbb{K}$. In this case, only increasing the number of collocation points and spectral coefficients is allowed. On the other hand, if $\left|\bm{\Omega}_k \right| \ge \epsilon_{ES}$, for some $k \in \mathbb{K}$, and the computational algorithm locates some edge points that are sufficiently close from the endpoints of $\bm{\Omega}_k$ in the sense that the distance between each point of them and an endpoint of $\bm{\Omega}_k$ is less than the prescribed $\epsilon_{ES}$, then these located edge points are to be discarded, and we break the interval at the remaining located edge points. If no other edge points exist, then we divide the domain using a similar partitioning technique to that adopted by the popular golden section search method. In particular, we break the interval at ${\tau _{k - 1}} + ({\tau _k} - {\tau _{k - 1}})/\varrho$, where $\varrho \approx 1.6180339887 \ldots$ represents the golden ratio.
\begin{rem}
If the state and control profiles are assumed to be sufficiently smooth, or the locations of discontinuities and/or ``points of nonsmoothness\footnote{The points at which a certain order derivative of a real-valued function is discontinuous.}'' are known a priori, one can easily apply the GISE method directly on the elements in which the solutions are sufficiently smooth without implementing the proposed adaptive strategy.
\end{rem}
\section{Error analysis of the KESOBGQ}
\label{subsec:err}
This section is devoted for analyzing the truncation error of the KESOBGQ constructed through interpolation at the adjoint KESGG nodes, since it constitutes a crucial numerical tool in the discretization procedure. 

Let ${\left\| g \right\|_{\infty ,{{\bm{\Omega }}_k}}} = \sup \left\{ {\left| {g(x)} \right|:x \in {{\bm{\Omega }}_k}} \right\}$, for any real-valued function $g$ defined on ${{\bm{\Omega }}_k} $, for all $ k \in \mathbb{K}$. Let also $\mathbb{P}{_{n}}$ denote the space of all polynomials of degree at most $n$, for some $n \in \mathbb{Z}^+$. The following theorem highlights the truncation error of the KESOBGQ, after successfully locating the discontinuities or the points of nonsmoothness.

\begin{thm}\label{sec:erranalysgp1}
Let $n_k, m_k \in \mathbb{Z}_0^+$, and consider any arbitrary integration nodes ${y_i^{(k)}} \in \bm{\Omega}_k, i = 0, \ldots, n_k$, for all $ k \in \mathbb{K}$. Furthermore, let $g\left(\tau^{(k)}\right) \in C^{m_k + 1}(\bm{\Omega}_k)$, be a real-valued function approximated by the truncated $k$th element shifted Gegenbauer polynomials expansion series 
such that the $k$th element shifted Gegenbauer spectral coefficients are computed at each node ${y_i^{(k)}}$ by interpolating the function $g$ at the adjoint KESGG nodes $\hat z_{{m_k},i,j}^{(k),\alpha _i^{(k),*}} \in \mathbb{T}_{n_k-1,m_k}$. Then there exist some real numbers $\xi_i^{(k)} \in (\tau_{k-1}, \tau_k), i =0, \ldots, n_k$, such that
\begin{equation}
\int_{{\tau _{k - 1}}}^{{y_i^{(k)}}} {g\left( {{\tau ^{(k)}}} \right)\,d{\tau ^{(k)}}}  = \sum\limits_{j = 0}^{{m_k}} {_k{p}_{OB,i,j}^{(1)}\,g\left( {\hat z_{{m_k},i,j}^{(k),\alpha _i^{(k),*}}} \right)}  + E_{{m_k}}^{\left( {\alpha _i^{(k),*}} \right)}\left( {y_i^{(k)},\xi _i^{(k)}} \right) ,
\end{equation}
where 
\begin{equation}
	E_{{m_k}}^{\left( {\alpha _i^{(k),*}} \right)}\left( {y_i^{(k)},\xi _i^{(k)}} \right) = \frac{{{g^{({m_k} + 1)}}\left( {\xi _i^{(k)}} \right)}}{{({m_k} + 1)!{\mkern 1mu} K_{k,{m_k} + 1}^{\left( {\alpha _i^{(k),*}} \right)}}}\int_{{\tau _{k - 1}}}^{y_i^{(k)}} {\hat G_{{m_k} + 1,k}^{\left( {\alpha _i^{(k),*}} \right)}\left( {{\tau ^{(k)}}} \right){\mkern 1mu} d{\tau ^{(k)}}} ,
\end{equation}
is the truncation error of the KESOBGQ, and
\begin{equation}
K_{k,j}^{(\alpha )} = \frac{{{2^{2j - 1}}}}{{{{\left( {{\tau _k} - {\tau _{k - 1}}} \right)}^j}}}\frac{{\Gamma \left( {2\alpha  + 1} \right)\Gamma \left( {j + \alpha } \right)}}{{\Gamma \left( {\alpha  + 1} \right)\Gamma \left( {j + 2\alpha } \right)}}\quad \forall j \in \mathbb{Z}_0^+,
\end{equation}
is the leading coefficient of the $j$th-degree $k$th element shifted Gegenbauer polynomial.
\end{thm}
\begin{proof}
For each node ${y_i^{(k)}}$, set the error term of the $k$th element shifted Gegenbauer interpolation as
\begin{equation}
	{R_{m_k,i}}\left(\tau^{(k)}\right) = g\left(\tau^{(k)}\right) - {P_{m_k,i}}g\left(\tau^{(k)}\right)\quad \forall i,
\end{equation}
where 
\begin{equation}
	{P_{m_k,i}}g\left(\tau^{(k)}\right) = \sum\limits_{j = 0}^{{m_k}} {a_{i,j}^{(k)} \hat G_{j,k}^{\left(\alpha _i^{(k),*}\right)}\left( {\tau^{(k)}} \right)}\quad \forall i,
\end{equation}
is the $m_k$th-degree $k$th element shifted Gegenbauer interpolant of the function $g$ on $\bm{\Omega}_k$, and ${a_{i,j}^{(k)}}, j = 0, \ldots, m_k$ are the $k$th element shifted Gegenbauer spectral coefficients. Now construct the auxiliary function
\begin{equation}
	{Y_i^{(k)}}(t) = {R_{m_k,i}}(t) - \frac{{{R_{m_k,i}}\left(\tau^{(k)}\right)}}{\hat G_{m_k+1,k}^{\left(\alpha _i^{(k),*}\right)}\left( {{\tau ^{(k)}}} \right)} \hat G_{m_k+1,k}^{\left(\alpha _i^{(k),*}\right)}\left( {{t}} \right)\quad \forall i.
\end{equation}
Since $g\left(\tau^{(k)}\right) \in C^{m_k + 1}(\bm{\Omega}_k)$, and ${P_{m_k,i}}g \in {C^\infty }(\bm{\Omega}_k),$ it follows that ${Y_i^{(k)}} \in {C^{m_k + 1}}(\bm{\Omega}_k)$. For $t = \hat z_{{m_k},i,j}^{(k),\alpha _i^{(k),*}}$, we have
\begin{equation}
	{Y_i^{(k)}}\left( \hat z_{{m_k},i,j}^{(k),\alpha _i^{(k),*}} \right) = {R_{{m_k,i}}}\left( \hat z_{{m_k},i,j}^{(k),\alpha _i^{(k),*}} \right) - \frac{{{R_{{m_k,i}}}\left( {{\tau ^{(k)}}} \right)}}{{\hat G_{{m_k} + 1,k}^{\left(\alpha _i^{(k),*}\right)}\left( {{\tau ^{(k)}}} \right)}}\hat G_{{m_k} + 1,k}^{\left(\alpha _i^{(k),*}\right)}\left( \hat z_{{m_k},i,j}^{(k),\alpha _i^{(k),*}} \right) = 0,
\end{equation}
since $\hat z_{{m_k},i,j}^{(k),\alpha _i^{(k),*}}, i = 0, \ldots, m_k$, are zeroes of ${{R_{m_k,i}}\left({{\tau ^{(k)}}}\right)}$. Moreover,
\begin{equation}
	{Y_i^{(k)}}\left( {{\tau ^{(k)}}} \right) = {R_{{m_k,i}}}\left( {{\tau ^{(k)}}} \right) - \frac{{{R_{{m_k,i}}}\left( {{\tau ^{(k)}}} \right)}}{{\hat G_{{m_k} + 1,k}^{\left(\alpha _i^{(k),*}\right)}\left( {{\tau ^{(k)}}} \right)}}\hat G_{{m_k} + 1,k}^{\left(\alpha _i^{(k),*}\right)}\left( {{\tau ^{(k)}}} \right) = 0.
\end{equation}
Thus ${Y_i^{(k)}} \in {C^{m_k + 1}}(\bm{\Omega}_k),$ and ${Y_i^{(k)}}$ is zero at the $(m_k + 2)$ distinct nodes $\tau^{(k)}, \hat z_{{m_k},i,j}^{(k),\alpha _i^{(k),*}}, j = 0, \ldots, m_k$. By the generalized Rolle's Theorem, there exists a number ${\xi_i}^{(k)}$ in $(\tau_{k-1}, \tau_k)$ such that ${\left({Y_i^{(k)}}\right)^{(m_k + 1)}}\left({\xi_i}^{(k)} \right) = 0$. Therefore,
\begin{align}
0 &= {\left({Y_i^{(k)}}\right)^{({m_k} + 1)}}\left({\xi_i}^{(k)} \right) = R_{{m_k,i}}^{({m_k} + 1)}\left({\xi_i}^{(k)} \right) - \frac{{{R_{{m_k,i}}}\left( {{\tau ^{(k)}}} \right)}}{{\hat G_{{m_k} + 1,k}^{\left(\alpha _i^{(k),*}\right)}\left( {{\tau ^{(k)}}} \right)}}\frac{{{d^{{m_k} + 1}}}}{{d{t^{{m_k} + 1}}}}{\left. {\hat G_{{m_k} + 1,k}^{\left(\alpha _i^{(k),*}\right)}(t)} \right|_{t = {\xi_i}^{(k)} }}\nonumber\\
&= R_{{m_k,i}}^{({m_k} + 1)}\left({\xi_i}^{(k)} \right) - \frac{{{R_{{m_k,i}}}\left( {{\tau ^{(k)}}} \right)}}{{\hat G_{{m_k} + 1,k}^{\left(\alpha _i^{(k),*}\right)}\left( {{\tau ^{(k)}}} \right)}}{\left( {\frac{2}{{{\tau _k} - {\tau _{k - 1}}}}} \right)^{{m_k} + 1}}\frac{{{d^{{m_k} + 1}}}}{{d{t^{{m_k} + 1}}}}{\left. {G_{{m_k} + 1}^{\left(\alpha _i^{(k),*}\right)}(t)} \right|_{t = {\xi_i}^{(k)} }}\nonumber\\
&= R_{{m_k,i}}^{({m_k} + 1)}\left({\xi_i}^{(k)} \right) - {\left( {\frac{2}{{{\tau _k} - {\tau _{k - 1}}}}} \right)^{{m_k} + 1}}{\mkern 1mu} ({m_k} + 1)!{\mkern 1mu} K_{{m_k} + 1}^{\left(\alpha _i^{(k),*}\right)}\frac{{{R_{{m_k,i}}}\left( {{\tau ^{(k)}}} \right)}}{{\hat G_{{m_k} + 1,k}^{\left(\alpha _i^{(k),*}\right)}\left( {{\tau ^{(k)}}} \right)}},
\end{align}
where $K_{{j}}^{\left(\alpha _i^{(k),*}\right)}$ is the leading coefficient of the Gegenbauer polynomial, $G_j^{\left(\alpha _i^{(k),*}\right)}\left( {{\tau ^{(k)}}} \right)$, for all $ j \in \mathbb{Z}_0^+$. Since ${P_{m_k,i}}g \in {\mathbb{P}_{m_k}},\,{({P_{m_k,i}}g)^{(m_k + 1)}}\left(\tau^{(k)}\right)$ is identically zero, then we have
\begin{align}
	&0 = {g^{({m_k} + 1)}}\left({\xi_i}^{(k)} \right) - {\left( {\frac{2}{{{\tau _k} - {\tau _{k - 1}}}}} \right)^{{m_k} + 1}}{\mkern 1mu} ({m_k} + 1)!{\mkern 1mu} K_{{m_k} + 1}^{\left(\alpha _i^{(k),*}\right)}\frac{{{R_{{m_k,i}}}\left( {{\tau ^{(k)}}} \right)}}{{\hat G_{{m_k} + 1,k}^{\left(\alpha _i^{(k),*}\right)}\left( {{\tau ^{(k)}}} \right)}}\\
&\Rightarrow g\left( {{\tau ^{(k)}}} \right) = {P_{{m_k},i}}g\left( {{\tau ^{(k)}}} \right) + {\left( {\frac{{{\tau _k} - {\tau _{k - 1}}}}{2}} \right)^{{m_k} + 1}}{\mkern 1mu} \frac{{{g^{({m_k} + 1)}}\left( {{\xi _i}^{(k)}} \right)}}{{({m_k} + 1)!{\mkern 1mu} K_{{m_k} + 1}^{\left( {\alpha _i^{(k),*}} \right)}}}\hat G_{{m_k} + 1,k}^{\left( {\alpha _i^{(k),*}} \right)}\left( {{\tau ^{(k)}}} \right).\label{eq:signifres1}
\end{align}
The proof of the theorem is established by rewriting the $m_k$th-degree $k$th element shifted Gegenbauer interpolant, ${P_{m_k,i}}g\left(\tau^{(k)}\right)$, in its equivalent Lagrange form (nodal approximation), and integrating both sides of Eq. \eqref{eq:signifres1} on the interval $[\tau_{k-1}, y_i^{(k)}]$.
\end{proof}

The following theorem marks the error bounds of the quadrature truncation error given by the above theorem on each element $\bm{\Omega}_k $, for all $ k \in \mathbb{K}$.
\begin{thm}\label{thm:2}
Given the assumptions of Theorem \ref{sec:erranalysgp1} such that ${\left\| {{g^{({m_k} + 1)}}} \right\|_{\infty ,{{\bm{\Omega }}_k}}} = {A_k} \in \mathbb{R}_0^ +$, for all $ k \in \mathbb{K}$, where the constant $A_k$ is dependent on $k$ but independent of $m_k$. Then there exist some constants $B_{1,k}^{\left( {\alpha _i^{(k),*}} \right)}$ and $B_2^{\left( {\alpha _i^{(k),*}} \right)}$, dependent on ${\alpha _i^{(k),*}}$ and independent of $m_k$ such that the quadrature truncation error, $E_{{m_k}}^{\left( {\alpha _i^{(k),*}} \right)}\left( {y_i^{(k)},\xi _i^{(k)}} \right)$, on each element $\bm{\Omega}_k$ is bounded by
\begin{equation}
	\begin{array}{l}
	{\left\| {E_{{m_k}}^{\left( {\alpha _i^{(k),*}} \right)}\left( {y_i^{(k)},\xi _i^{(k)}} \right)} \right\|_{\infty ,{\bm{\Omega}_k}}} =  B_{1,k}^{\left( {\alpha _i^{(k),*}} \right)}\,{2^{ - 2{m_k} - 1}}{{{e}}^{{m_k}}}{m_k}^{\alpha _i^{(k),*} - {m_k} - \frac{3}{2}}\left( {y_i^{(k)} - {\tau _{k - 1}}} \right){\left( {{\tau _k} - {\tau _{k - 1}}} \right)^{{m_k} + 1}} \times \\
	\left( {\left\{ \begin{array}{l}
	1,\quad {m_k} \ge 0 \wedge \alpha _i^{(k),*} \ge 0,\\
	\displaystyle{\frac{{\Gamma \left( {\frac{{{m_k}}}{2} + 1} \right)\Gamma \left( {\alpha _i^{(k),*} + \frac{1}{2}} \right)}}{{\sqrt \pi  \Gamma \left( {\frac{{{m_k}}}{2} + \alpha _i^{(k),*} + 1} \right)}}},\quad \frac{{{m_k} + 1}}{2} \in {\mathbb{Z}^ + } \wedge  - \frac{1}{2} < \alpha _i^{(k),*} < 0,\\
	\displaystyle{\frac{{2\Gamma \left( {\frac{{{m_k} + 3}}{2}} \right)\Gamma \left( {\alpha _i^{(k),*} + \frac{1}{2}} \right)}}{{\sqrt \pi  \sqrt {\left( {{m_k} + 1} \right)\left( {{m_k} + 2\alpha _i^{(k),*} + 1} \right)} \Gamma \left( {\frac{{{m_k} + 1}}{2} + \alpha _i^{(k),*}} \right)}}},\quad \frac{{{m_k}}}{2} \in \mathbb{Z}_0^ +  \wedge  - \frac{1}{2} < \alpha _i^{(k),*} < 0,\\
	B_2^{\left( {\alpha _i^{(k),*}} \right)} {\left( {{m_k} + 1} \right)^{ - \alpha _i^{(k),*}}},\quad {m_k} \to \infty  \wedge  - \frac{1}{2} < \alpha _i^{(k),*} < 0
	\end{array} \right.} \right),
	\end{array}
\end{equation}
where $B_{1,k}^{\left( {\alpha _i^{(k),*}} \right)} = {A_k}{D^{\left( {\alpha _i^{(k),*}} \right)}}$; the constants ${D^{\left( {\alpha _i^{(k),*}} \right)}} > 0$ and $B_2^{\left( {\alpha _i^{(k),*}} \right)} > 1$ are dependent on ${\alpha _i^{(k),*}}$, but independent of $m_k$.
\end{thm}
\begin{proof}
The proof can be established easily using \cite[Lemmas 4.1 and 4.2]{Elgindy2016}.
\end{proof}

\begin{rem}
It is noteworthy to mention that the accuracy achieved by a spectral differentiation/integration matrix used by traditional pseudospectral methods in the literature is usually constrained by the number of collocation points. Therefore, increasing the number of collocation points on each domain $\bm{\Omega}_k$, requires a similar grow in the size of the spectral differentiation/integration matrix, which could result in a significant grow in the total computational cost of the method. On the other hand, a notable merit of the present method as shown by Theorem \ref{thm:2} occurs in taking advantage of the free rectangular form of the KESOBGIM. In particular, regardless of the small/large number of collocation points used to determine the approximate states and controls, the present method endowed with the KESOBGIM can achieve almost full machine precision approximations to the integrals involved in the optimal control problem using relatively moderate values of the parameters $M_k$ and $\bar M_k$ on each element $\bm{\Omega_k}$; thus achieving excellent approximations while maintaining a low operational cost. We shall demonstrate this virtue further in the next section.
\end{rem}


%
\section{Numerical examples}
\label{sec:ne}
In this section, we report the results of the present GISE method on three nonlinear optimal control problems well studied in the literature. The nonlinear programming problems were solved using SNOPT software \cite{snopt75,GilMS05} with the major and minor feasibility tolerances, and major optimality tolerance all set at $10^{-10}$. The numerical experiments were conducted on a personal laptop equipped with an Intel(R) Core(TM) i7-2670QM CPU with 2.20GHz speed running on a Windows 10 64-bit operating system and provided with MATLAB R2014b (8.4.0.150421) software. 

\paragraph{\textbf{Example 1}} Consider the following nonlinear optimal control problem:
\begin{subequations}\label{Ex1:1}
\begin{equation}
	{\text{Minimize }}J = \int_0^1 {\sin (3\,\pi \,t)\,x(t)\,dt}
\end{equation}
\text{subject to } 
\begin{equation}
	\dot x(t) =  - \tan \left( {\frac{\pi }{8}{u^3}(t) + t} \right),\quad t \in [0,1],
\end{equation}
\begin{equation}
	u(t) \in [0,1],\;x(0) = 1,\;x(1) = 0.
\end{equation}
\end{subequations}

This problem was numerically solved in a series of papers; cf. \cite{von1993,Skandari2011,Tohidi2013}. The work presented in the latest article of this series, Ref. \cite{Tohidi2013}, adopted a linearization of the nonlinear dynamical system via a linear combination property of intervals followed by a ``\textit{random}'' interval partitioning (three switching points were chosen randomly) and an integral reformulation of the multidomain dynamical system. The transformed problem was then collocated at the Legendre-Gauss-Lobatto points and the involved integrals were approximated using the Legendre-Gauss-Lobatto quadrature rule. Moreover, the control and state functions were approximated by piecewise constants and piecewise polynomials, respectively. 

We applied the present GISE method for solving the problem numerically using the parameter settings $N_k = 14, L_{x,k} = L_{u,k} = 8, M_k = 16, \bar M_k = 4, \bar N_k = \bar L_{x,k} = \bar L_{u,k} = 4, N_{k,\max} = L_{x,k,\max} = L_{u,k,\max} = 20$, for all $ k \in \mathbb{K}, \alpha = 0.2, \epsilon_{\mathbf{R}} = 10^{-2}, \epsilon_{\text{coeff}} = 10^{-1}, \rho = 3, k_{\max} = 20$, and $\epsilon_{ES} = 0.1$. All state and control coefficients were initially set to one. Figure \ref{fig:Ex1StateControl} shows a sketch of the obtained approximate optimal state and control profiles on $[0,1]$ with a reported approximate optimal cost function value of $0.104$. In contrast with the work of \cite{Tohidi2013}, the present GISE method endowed with the proposed adaptive strategy finds no points of discontinuities/nonsmoothness, as the state and control functions appear to be sufficiently smooth on the interval $[0,1]$. Therefore, the adaptivity of the method enables a fast implementation using a single collocation grid without any domain partitioning; thus $K = 1$. Figure \ref{fig:Ex1Fun1} shows the plot of the approximate optimal cost functional $J_{N,L_x,L_u}^{(\alpha),*}$, for several values of $N_1, L_{x,1}, L_{u,1},$ and $\alpha$. As observed from the figure, the reported approximate optimal cost function values approach $0.1$ for increasing values of collocation points and spectral coefficients in a close agreement with the results obtained by \cite{von1993,Skandari2011,Tohidi2013}. Figure \ref{fig:Ex1Coeff1} manifests further the corresponding exponential (spectral) decay of the last optimal coefficients in the state and control truncated series, $\left| {a_{1,{L_{x,1}}}^{(1),*}} \right|$ and $\left| {b_{1,{L_{u,1}}}^{(1),*}} \right|$.  

In fact, Figure \ref{fig:Ex1Coeff1} shows an interesting behavior of the GISE method. In particular, the figure shows that the last optimal coefficients in the state and control shifted Gegenbauer truncated series generally decay faster for negative values of the Gegenbauer parameter $\alpha$ than for positive values, and this deterioration phenomenon seems to happen monotonically as the value of $\alpha$ approaches $-0.5$. This numerical simulation is in close consensus with the work of \cite{Elgindy2016} on the numerical solution of the second-order one-dimensional hyperbolic telegraph equation using a shifted Gegenbauer pseudospectral method. In particular, the latter showed theoretically that the coefficients of the bivariate shifted Gegenbauer expansions decay faster for negative $\alpha$-values than for non-negative $\alpha$-values, but the asymptotic truncation error as the number of collocation points grows largely is minimized in the infinity norm (Chebyshev norm) exactly at $\alpha = 0$; i.e., when applying the shifted Chebyshev basis polynomials. Figure \ref{fig:Ex1Fun1} indicates that collocations at negative values of $\alpha$ close to $-0.5$ is not to be endorsed for increasing values of collocation points and expansion terms. In particular, while the values of $J_{N,L_x,L_u}^{(\alpha),*}$ seem to be matching for almost all of the $\alpha$-values used in the numerical simulation, a peak in the surface of the approximate optimal cost function is clearly observed at $\alpha = -0.4$, for $N_1 = 16$, indicating a poor approximation in this case. On the other hand, \cite{Elgindy2013} pointed out that the Gegenbauer quadrature `may become sensitive to round-off errors for positive and large values of the parameter $\alpha$ due to the narrowing effect of the Gegenbauer weight function,' which drives the quadrature to become more extrapolatory. In particular, \cite{Elgindy2013} identified the range $-1/2 + \varepsilon \le \alpha \le r$, as a preferable choice to construct the Gegenbauer quadrature, for some relatively small positive number $\varepsilon$ and $r \in [1,2]$. We shall refer to the interval $[-1/2 + \varepsilon, r]$ by \textit{``the Gegenbauer collocation interval of choice,''} and denote it by ${I_{\varepsilon,r}^G}$. Figure \ref{Ex1Fun1_ExpandedPerfect} shows a sketch of the approximate optimal cost functional $J^{(\alpha),*}_{N,L_x,L_u}$ for $N_1 = 6(2)16, L_{x,1} = L_{u,1} = \left\lceil {N/2} \right\rceil  + 1$, and $\alpha = 1(0.5)10$, where we can clearly see the rise of hills in the surface profile for increasing values of $\alpha \notin {I_{\varepsilon,r}^G}$ demonstrating poor approximations for such $\alpha$-values. This formation of hills for increasing values of $\alpha \notin {I_{\varepsilon,r}^G}$ is salient as well in the surface profiles of the corresponding magnitudes of the last coefficients in the state and control truncated series; cf. Figure \ref{Ex1Coeff1ExpandedPerfect}. In general, we largely endorse the following rule of thumb.

\paragraph{\textbf{Rule of Thumb}} \textit{It is generally advantageous to collocate Problem 3 for values of $\alpha \in {I_{\varepsilon,r}^G}$ for small/medium numbers of collocation points and Gegenbauer expansion terms; however, collocations at the shifted Chebyshev-Gauss points should be put into effect for large numbers of collocation points and Gegenbauer expansion terms if the approximations are sought in the Chebyshev norm.}\\

We shall further examine experimentally this rule of thumb in the next example, where the exact control function is given in closed form.

\begin{figure}[ht]
\centering
\includegraphics[scale=0.7]{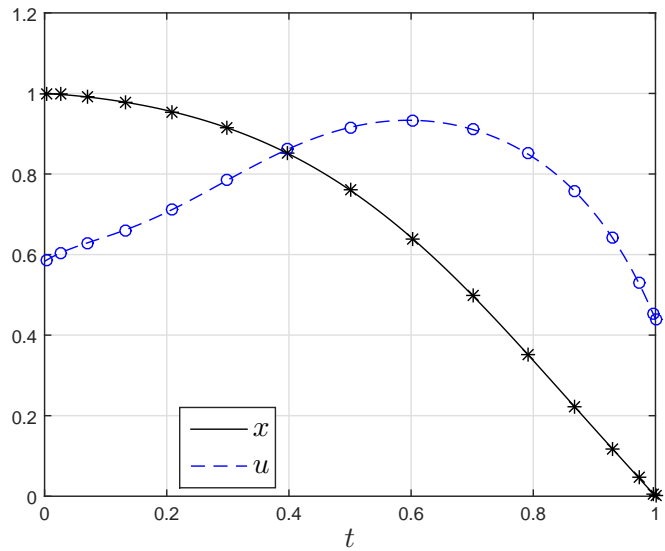}
\caption{The figure shows the plot of the approximate state and control functions of Example 1 on $[0,1]$ using $N_1 = 14, L_{x,1} = L_{u,1} = 8$, and $\alpha = 0.2$. The plot was generated using $100$ linearly spaced nodes from $0$ to $1$.}
\label{fig:Ex1StateControl}
\end{figure}
\begin{figure}[ht]
\centering
\includegraphics[scale=0.75]{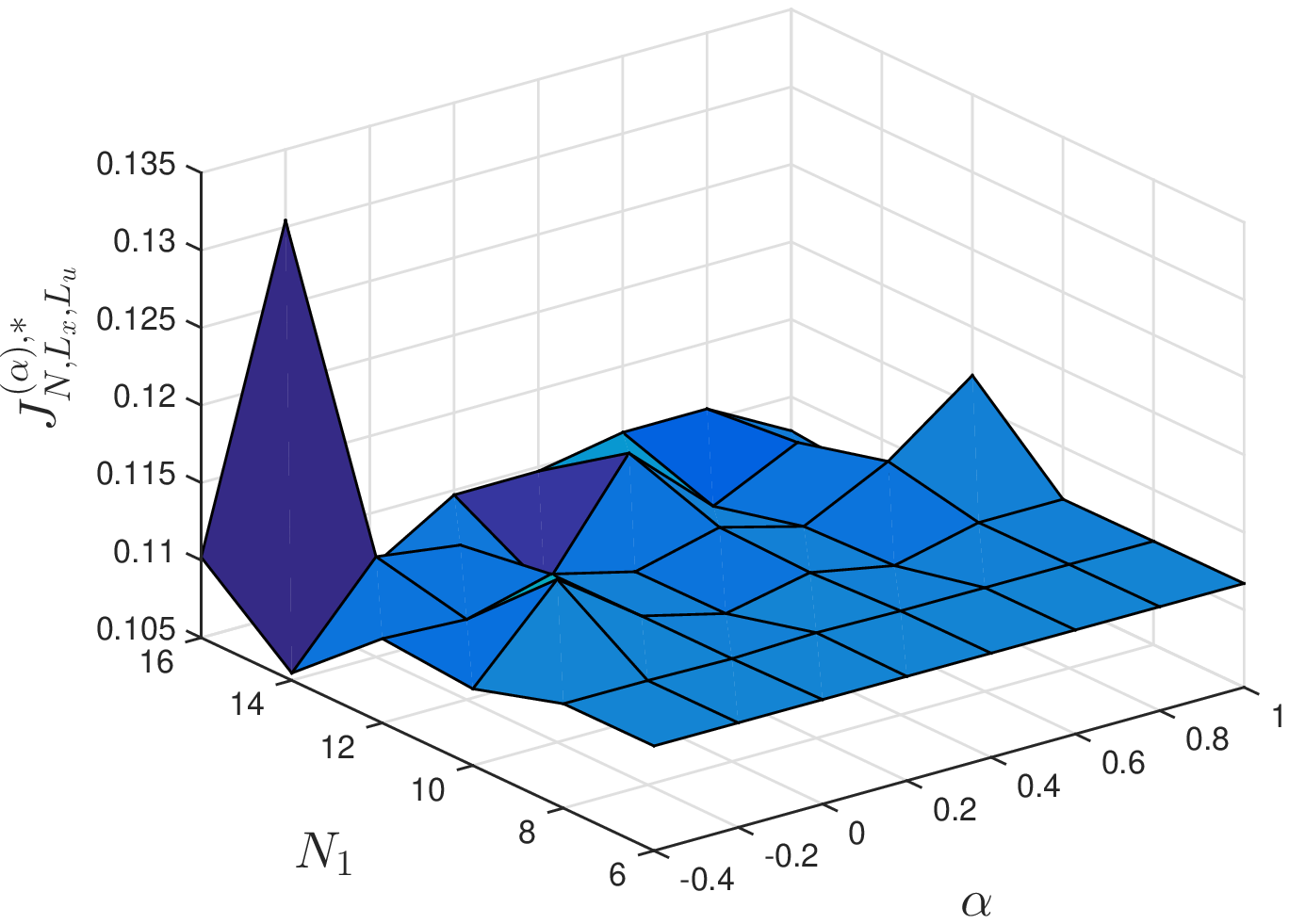}
\caption{The figure shows the plot of the approximate optimal cost functional $J^{(\alpha),*}_{N,L_x,L_u}$ of Example 1 for $N_1 = 6(2)16, L_{x,1} = L_{u,1} = \left\lceil {N/2} \right\rceil  + 1$, and $\alpha = -0.4(0.2)1$.}
\label{fig:Ex1Fun1}
\end{figure}
\begin{figure}[ht]
\centering
\includegraphics[scale=0.85]{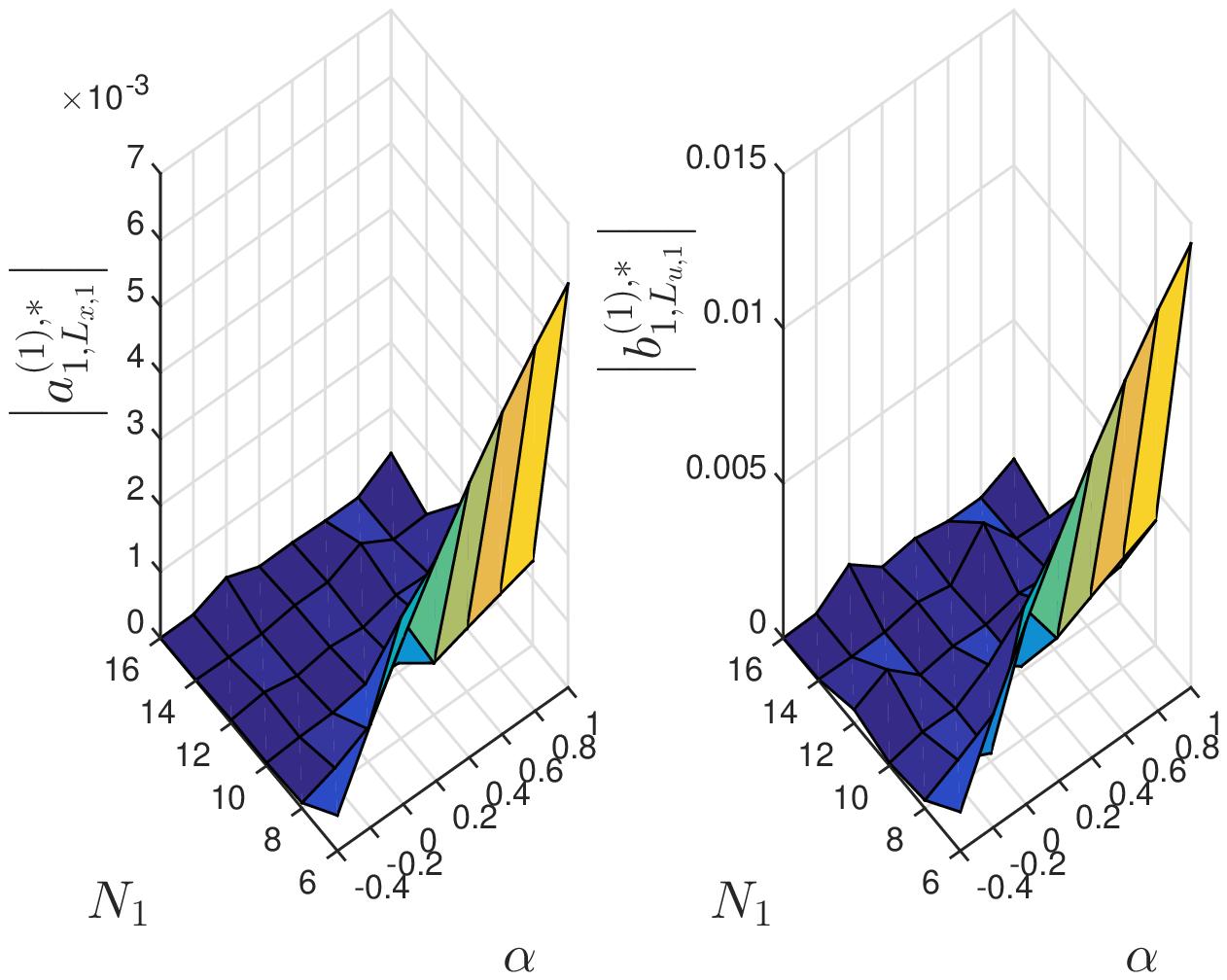}
\caption{The figure shows the magnitudes of the last coefficients in the state and control truncated series of Example 1 for $N_1 = 6(2)16, L_{x,1} = L_{u,1} = \left\lceil {N/2} \right\rceil  + 1$, and $\alpha = -0.4(0.2)1$.}
\label{fig:Ex1Coeff1}
\end{figure}
\begin{figure}[ht]
\centering
\includegraphics[scale=0.55]{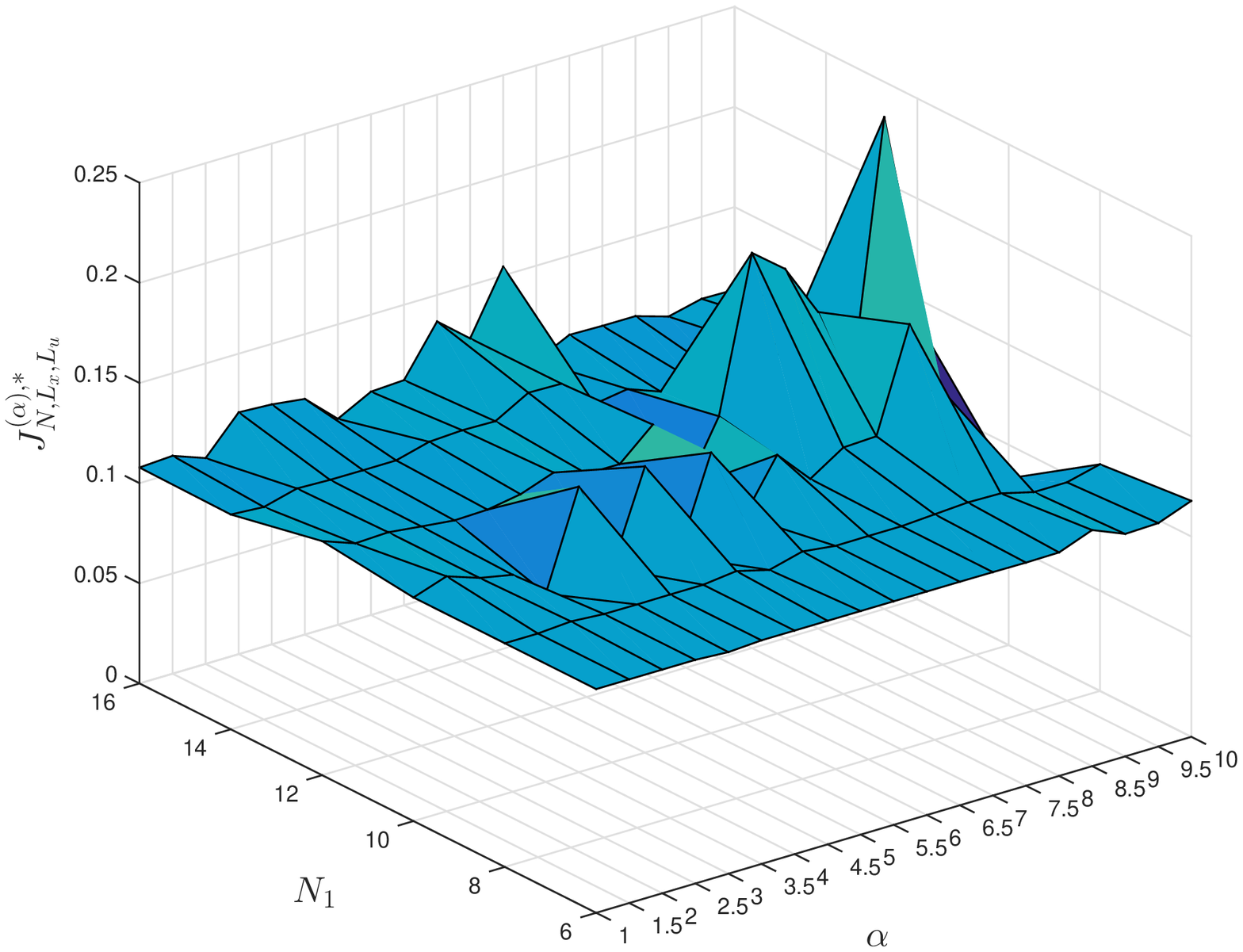}
\caption{The figure shows the plot of the approximate optimal cost functional $J^{(\alpha),*}_{N,L_x,L_u}$ of Example 1 for $N_1 = 6(2)16, L_{x,1} = L_{u,1} = \left\lceil {N/2} \right\rceil  + 1$, and $\alpha = 1(0.5)10$.}
\label{Ex1Fun1_ExpandedPerfect}
\end{figure}
\begin{figure}[ht]
\centering
\includegraphics[scale=0.35]{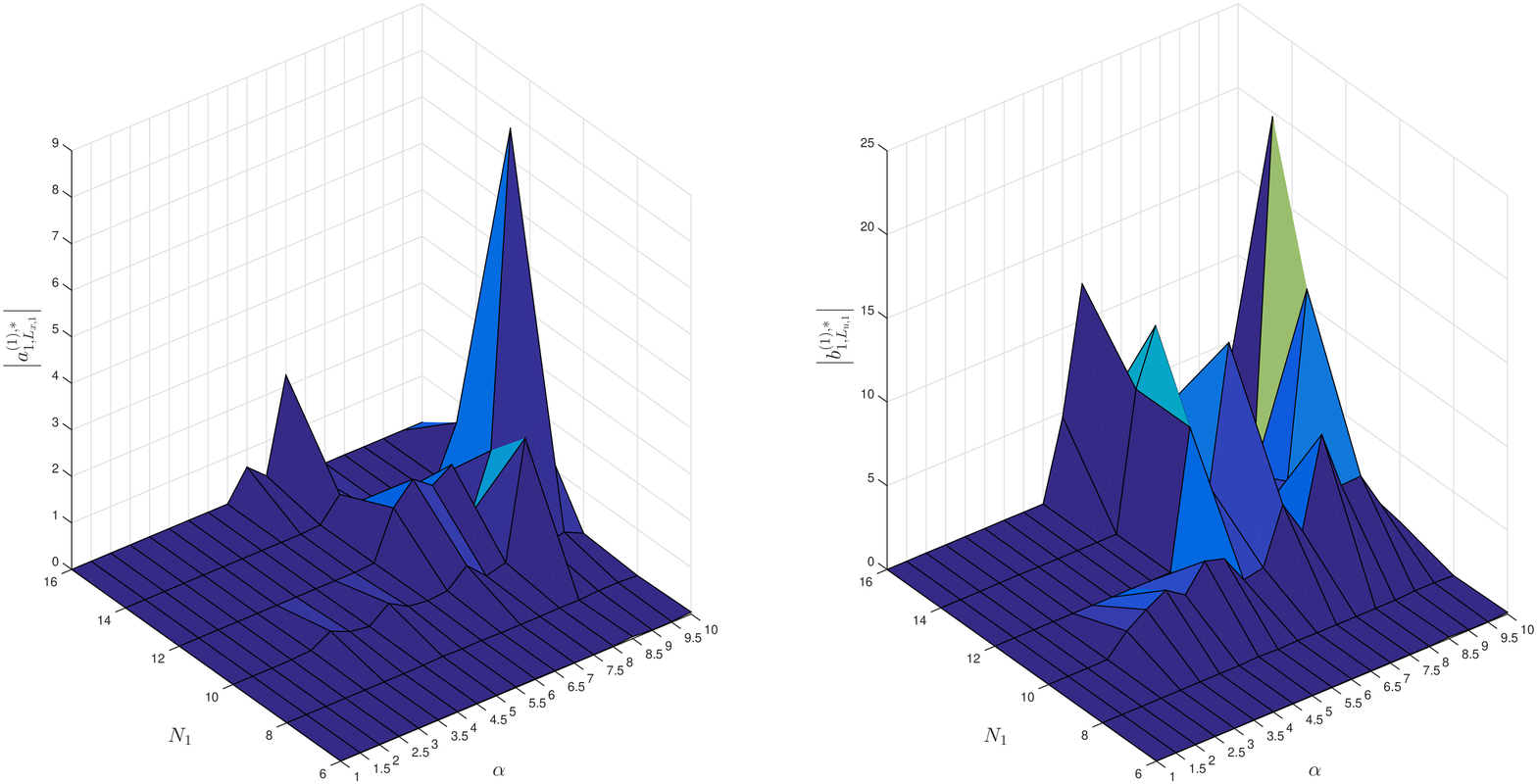}
\caption{The figure shows the magnitudes of the last coefficients in the state and control truncated series of Example 1 for $N_1 = 6(2)16, L_{x,1} = L_{u,1} = \left\lceil {N/2} \right\rceil  + 1$, and $\alpha = 1(0.5)10$.}
\label{Ex1Coeff1ExpandedPerfect}
\end{figure}

\paragraph{\textbf{Example 2}} Consider the popular Breakwell problem \cite{Ho1975,Gong2006,Marzban2013}:
\begin{subequations}\label{Ex2:2}
\begin{equation}
	{\text{Minimize }}J = \frac{1}{2}\int_0^1 {{u^2}(t)\,dt}
\end{equation}
	\text{subject to } 
\begin{equation}
	{{\dot x}_1}(t) = {x_2}(t),\quad t \in [0,1],
\end{equation}
\begin{equation}
	{{\dot x}_2}(t) = u(t),\quad t \in [0,1],
\end{equation}
\begin{equation}
	{x_1}(0) = {x_1}(1) = 0,
\end{equation}
\begin{equation}
	{x_2}(0) =  - {x_2}(1) = 1,
\end{equation}
\begin{equation}
	{x_1}(t) \le 0.1.
\end{equation}
\end{subequations}
The exact control function to this problem is given by
\begin{equation}
	{u^*}(t) = \left\{ \begin{array}{l}
200\,t/9 - 20/3,\quad t \in [0,0.3],\\
0,\quad t \in [0.3,0.7],\\
 - 200\,t/9 + 140/9,\quad t \in [0.7,1],
\end{array} \right.
\end{equation}
and the optimal cost function value is $J^* = 40/9$. This problem was numerically solved by \cite{Gong2006} using a direct Legendre pseudospectral, and the collocation was performed at the Legendre-Gauss-Lobatto quadrature nodes; however, the observed convergence rate was very slow due to the lack of smoothness of the optimal solution. Later, \cite{Marzban2013} numerically solved the problem using a direct composite Chebyshev finite difference method based on a hybrid of block-pulse functions and Chebyshev polynomials, and the implementation was carried out using Chebyshev-Gauss-Lobatto points. 

We implemented the present GSE for solving the problem numerically using the parameter settings $N_k = 18, L_{x,k} = L_{u,k} = 17, M_k = 16, \bar M_k = 4, \bar N_k = \bar L_{x,k} = \bar L_{u,k} = 4, N_{k,\max} = L_{x,k,\max} = L_{u,k,\max} = 30$, for all $ k \in \mathbb{K}, \alpha = 0.5, \epsilon_{\mathbf{R}} = 10^{-2}, \epsilon_{\text{coeff}} = 10^{-3}, \rho = 1.5, k_{\max} = 20$, and $\epsilon_{ES} = 0.1$. All state and control coefficients were initially set to zero. The computational algorithm breaks the transformed interval $[-1,1]$ into the three domains $\bm{\Omega}_1 = [-1, -0.3906], \bm{\Omega}_2 = [-0.3906,0.3906]$, and $\bm{\Omega}_3 = [0.3906,1]$; thus $K = 3$. These three mesh intervals correspond to the three domains $\bm{\Omega}_{1,\text{orig}} = [0, 0.3047], \bm{\Omega}_{2,\text{orig}} = [0.3047,0.6953]$, and $\bm{\Omega}_{3,\text{orig}} = [0.6953,1]$ of the original optimal control problem. Figure \ref{Ex2_All} shows the plots of the approximate state functions, $x_1(t)$ and $x_2(t)$, exact and approximate control functions, $u^*(t)$ and $u(t)$, respectively, on the interval $[0,1]$. The corresponding approximate optimal cost function value, $J^{(\alpha),*}_{N,L_x,L_u} \approx 4.444477 \approx J^*$. Figure \ref{Ex2_u_err} shows the absolute error of the control function values, $\left| {{u^*}(t) - u(t)} \right|$, on the interval $[0,1]$ in log-lin scale, where a rapid error decay is clearly seen. In contrast, Figure \ref{Ex2_u_err_single} shows the slow convergence of the absolute error on the interval $[0,1]$ in log-lin scale using the same parameter settings, but with a single collocation grid. 

\begin{rem}
We expect to attain faster convergence rates of the present GSE by optimizing the parameters $\rho, \epsilon_{\mathbf{R}}$, and $\epsilon_{\text{coeff}}$ relative to the initial inputs $N_k, L_{x,k}$, and $L_{u,k}$, for all $ k \in \mathbb{K}$, so that the determined edge points are sufficiently close from the discontinuities or points of nonsmoothness. However, we prefer to leave this interesting topic as a future research.
\end{rem}


It is also intriguing to study the numerical effect of different values of $\alpha$ on the numerical scheme. To this end, and for fairness of comparisons, we implemented the GISE method in the absence of the proposed adaptive strategy, and broke the transformed domain $[-1,1]$ at the exact edge points $-0.4$ and $0.4$, which correspond to the original edge points $0.3$ and $0.7$ in $[0,1]$. Figure \ref{Ex2Costn0p4to1} shows the plot of the approximate optimal cost functional $J^{(\alpha),*}_{N,L_x,L_u}$ for $N_k = M_k = 4, L_{x,k} = L_{u,k} = 3$, for all $ k \in \mathbb{K}$, and $\alpha = -0.4(0.1)1$-- such input values lead to a small-scale nonlinear programming problem of dimension $12$. Obviously, the cost functional profile decreases as $\alpha$ progresses away from $-0.5$, then remains nearly steady at about $4.444449$ till $\alpha$ reaches the value $1$. Figure \ref{Ex2Coeff0p4to1} shows the corresponding magnitudes of the last coefficients in the state and control truncated series on the mesh intervals $\bm{\Omega}_{k,\text{orig}}$, for all $ k \in \mathbb{K}$. On the other hand, Figure \ref{Ex2Cost1to20} shows a sketch of the approximate optimal cost functional profile for $\alpha = 1(1)20$, where we see a significant rise in the functional value reported at $\alpha = 9, 12, 16(1)18, 20$. This increase in the functional value is captured very clearly in the profile of the corresponding magnitudes of the last coefficients in the state and control truncated series on the mesh intervals $\bm{\Omega}_{k,\text{orig}}$, for all $ k \in \mathbb{K}$; cf. Figure \ref{Ex2Coeff1to20}. In particular, we can see some few jumps and increases in the last coefficients' magnitudes exactly at the reported $\alpha$-values. Therefore, for the given parameter settings, we can consider the set ${I_{0.4,2}^G}$ as a feasible Gegenbauer collocation interval of choice.
To investigate further the convergence of the GISE method for the same two sets of possible $\alpha$-values, we implemented the method using the relatively medium values $N_k = M_k = L_{x,k} = L_{u,k} = 14$, for all $ k \in \mathbb{K}$, which lead to a medium-scale nonlinear programming problem of dimension $45$. Figure \ref{Ex2Costn0p4to122} shows the plot of the approximate optimal cost functional $J^{(\alpha),*}_{N,L_x,L_u}$ for $\alpha = -0.4(0.1)1$, where a reduction in its value is generally observed for increasing values of $\alpha$ till it arrives at a nearly uniform state at about $4.44$ for $\alpha \ge -0.2$. Moreover, the corresponding magnitudes of the last coefficients in the state and control truncated series remain bounded below $10^{-5}$ and $10^{-8}$, respectively, on the mesh intervals $\bm{\Omega}_{k,\text{orig}}$, for all $ k \in \mathbb{K}$; cf. Figure \ref{Ex2Coeff0p4to122}. On the other hand, Figure \ref{Ex2Costn0p4to12244} shows a sketch of the approximate optimal cost functional profile for $\alpha = 1(1)20$, where we see some wild escalations in the functional value at $\alpha = 3$ and $11$ accompanied by another significant rise for $\alpha > 13$. This serious degradation in the accuracy of the functional value is also observed very clearly in the profile of the corresponding magnitudes of the last coefficients in the state and control truncated series on the mesh intervals $\bm{\Omega}_{1,\text{orig}}$, for all $ k \in \mathbb{K}$; in particular, the last coefficients' magnitudes generally soar rapidly as $\alpha$ increases; cf. Figure \ref{Ex2Coeff0p4to12244}. Hence, for the given parameter settings, the set ${I_{0.3,2}^G}$ is a feasible Gegenbauer collocation interval of choice. Figure \ref{fig:All} shows further the absolute errors of the control function values, $\left| {{u^*}(t) - u(t)} \right|$, on the interval $[0,1]$ in log-lin scale using 
$\alpha = -0.4(0.2)0.2,0.5,1,2,10,20$. High-order control approximations are achieved in all cases, except for $\alpha = -0.4, 10$, and $20$, where a significant recession in accuracy is reported near the time boundary points $t = 0$ and $t = 1$ in the former case, while a tangible error growth in the vicinity of the right endpoint $t = 1$ is reported in the latter two cases.

\begin{rem}
A typical hp-pseudospectral method would normally apply a square differentiation/integration matrix of size $N_k + 1$, for all $ k \in \mathbb{K}$ to compute the derivatives/integrals involved in the optimal control problem; thus requires $(1 + N_k) (1 + 2 N_k)$ FLOPS to evaluate the derivatives of a real-valued differentiable function at a set of $N_k + 1$ collocation points for each $k \in \mathbb{K}$, or the definite integrals of an integrable function using the same sets of collocation points as the upper limits of the integrations. In contrast, the KESOBGIM requires $(1 + 2 M_k) (1 + N_k)$ FLOPS to evaluate the needed integrals. To prevent an enormous amount of calculations, we can set $M_k$ at a relatively medium value, say, $16$-- usually sufficient to achieve nearly full machine precision approximations to the integrals of well-behaved functions-- for large values of $N_k$, for all $ k \in \mathbb{K}$. For instance, using $N_k = 100$ and $M_k = 16$, for all $ k \in \mathbb{K}$, we can evidently count a substantial difference of $16968$ FLOPS between the developed KESOBGIM and a standard operational matrix of differentiation/integration for each derivative/integral per mesh interval. To visualize the big picture, notice that the discretization of the present optimal control problem requires the evaluation of a single integral for the cost functional, and $\sum\nolimits_{k = 1}^3 {\left( {{N_k} + 2} \right)}$ integrals involved in Eqs. \eqref{eq:discdynsys1}, for a total of $307$ integral evaluations. Working out the mathematics, it is not hard to realize a remarkable gap of $5209176$ FLOPS in favor of the present GISE method endowed with the KESOBGIM!
\end{rem}

\begin{rem}
The current study casts the light on the judicious choice of the KESGG collocation points set, $\mathbb{S}_{N_k}^{(\alpha)}$, to be utilized on each mesh interval $\bm{\Omega}_k$, for all $ k \in \mathbb{K}$ during the discretization process of optimal control problems. In particular, the current work supports collocations performed at $\mathbb{S}_{N_k}^{(\alpha^{(k)})}: \alpha^{(k)} \in I_{\varepsilon,r}^G$, for small/medium numbers of collocation points and Gegenbauer expansion terms. Nonetheless, it would be extremely beneficial to determine theoretically the optimal collocation sets $\mathbb{S}_{N_k}^{\alpha^{(k),*}}$ for each domain $\bm{\Omega}_k$, for all $ k \in \mathbb{K}$-- \textbf{a question which yet remains open}. 
\end{rem}

\begin{figure}[ht]
\centering
\includegraphics[scale=0.7]{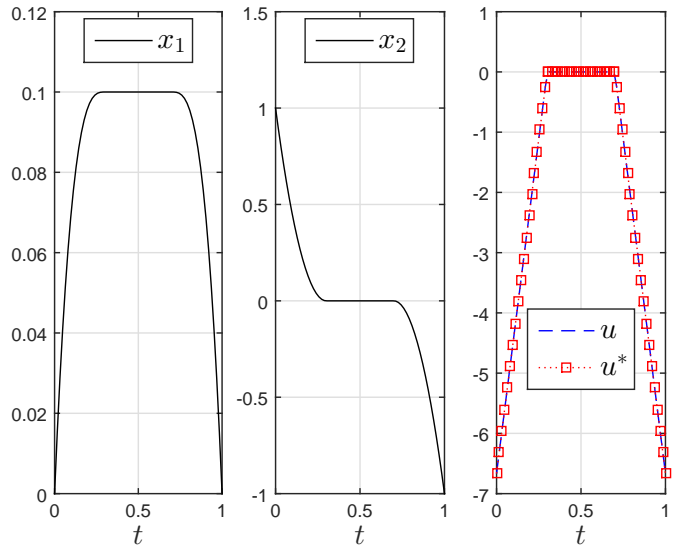}
\caption{The figure shows the plots of the approximate state functions, $x_1(t)$ and $x_2(t)$ (left and middle), exact and approximate control functions, $u^*(t)$ and $u(t)$ (right), respectively, of Example 2 on the interval $[0,1]$ using the initial values $N_k = 18, L_{x,k} = L_{u,k} = 17\;\forall k \in \mathbb{K}$, and $\alpha = 0.5$. The plots were generated using $20$ linearly spaced nodes in each domain $\bm{\Omega}_{k,\text{orig}}$, for all $ k \in \mathbb{K}$.}
\label{Ex2_All}
\end{figure}
\begin{figure}[ht]
\centering
\subfigure[]{
\includegraphics[scale=0.525]{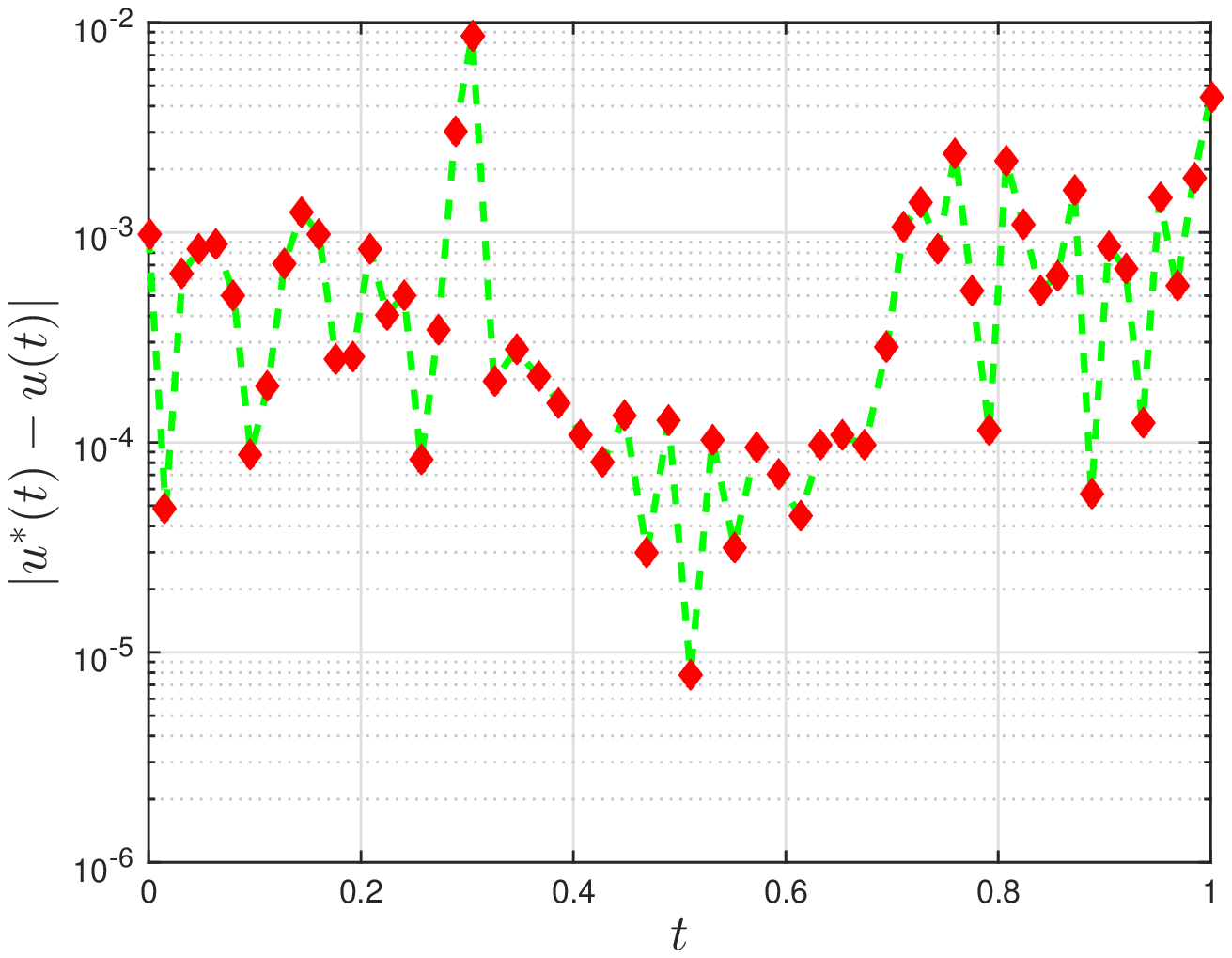}
\label{Ex2_u_err}
}
\subfigure[]{
\includegraphics[scale=0.525]{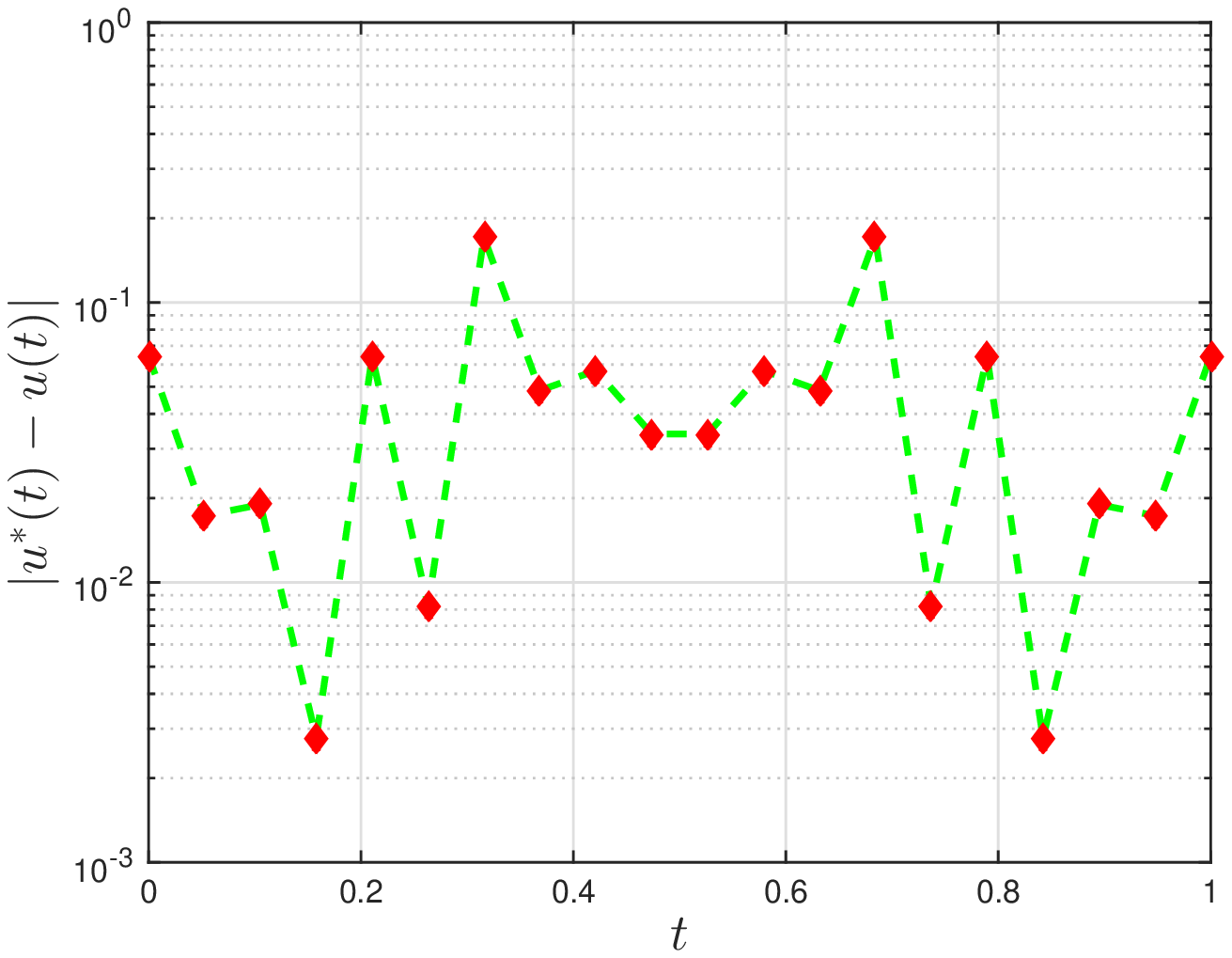}
\label{Ex2_u_err_single}
}
\caption{Figure \subref{Ex2_u_err} shows the absolute error of the control function values, $\left| {{u^*}(t) - u(t)} \right|$, of Example 2 on the interval $[0,1]$ in log-lin scale using the initial values $N_k = 18, L_{x,k} = L_{u,k} = 17$, for all $ k \in \mathbb{K}$, and $\alpha = 0.5$. The plot was generated using $20$ linearly spaced nodes in each domain $\bm{\Omega}_{k,\text{orig}}$, for all $ k \in \mathbb{K}$. Figure \subref{Ex2_u_err_single} shows the absolute error of the control function values $\left| {{u^*}(t) - u(t)} \right|$ on the interval $[0,1]$ in log-lin scale using a single collocation grid. The plot was generated using $20$ linearly spaced nodes in $[0,1]$.}
\label{fig:Ex2specific1}
\end{figure}
\begin{figure}[ht]
\centering
\includegraphics[scale=0.6]{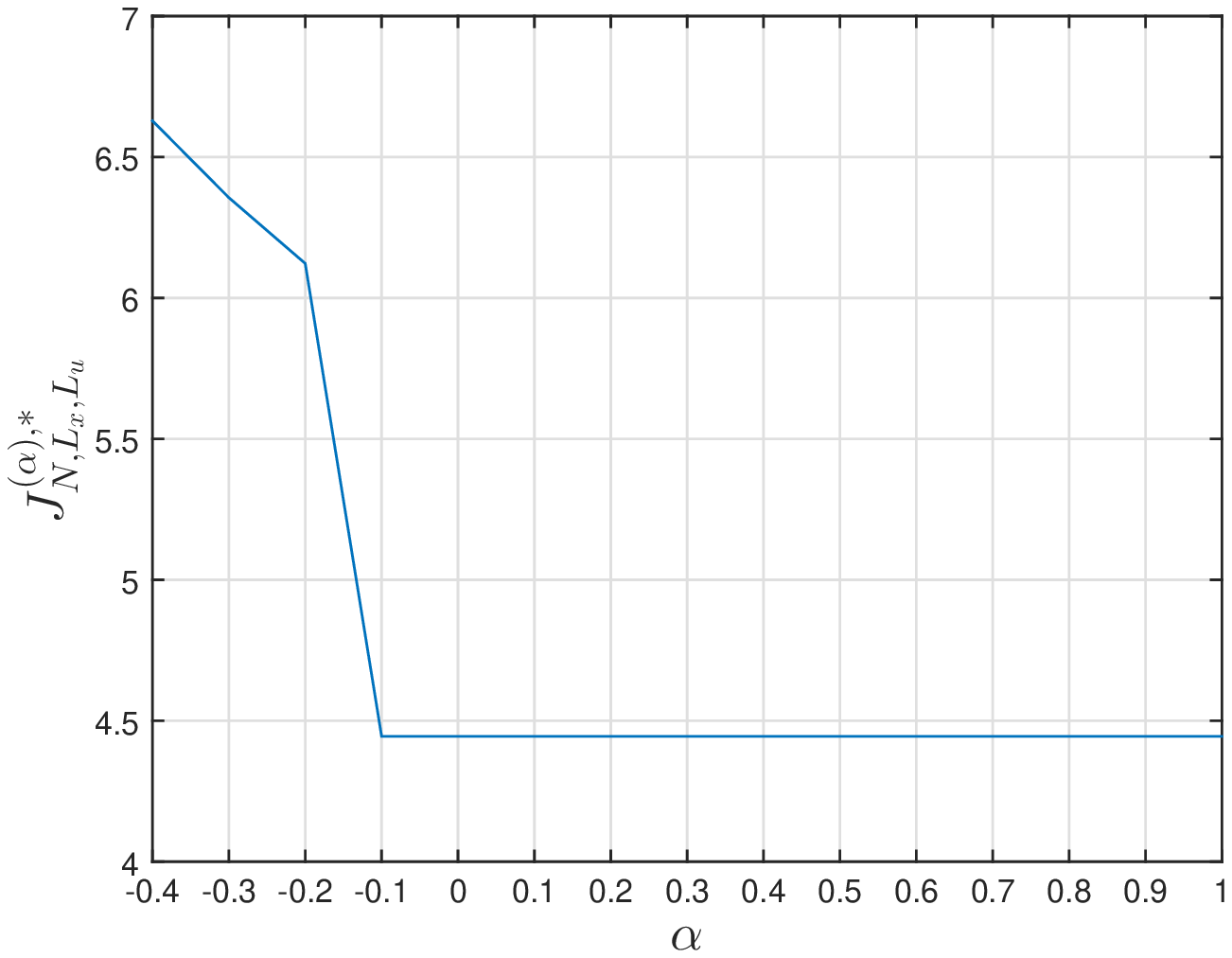}
\caption{The figure shows the plot of the approximate optimal cost functional $J^{(\alpha),*}_{N,L_x,L_u}$ of Example 2 for $N_k = M_k = 4, L_{x,k} = L_{u,k} = 3$, for all $ k \in \mathbb{K}$, and $\alpha = -0.4(0.1)1$.}
\label{Ex2Costn0p4to1}
\end{figure}
\begin{figure}[ht]
\centering
\includegraphics[scale=0.35]{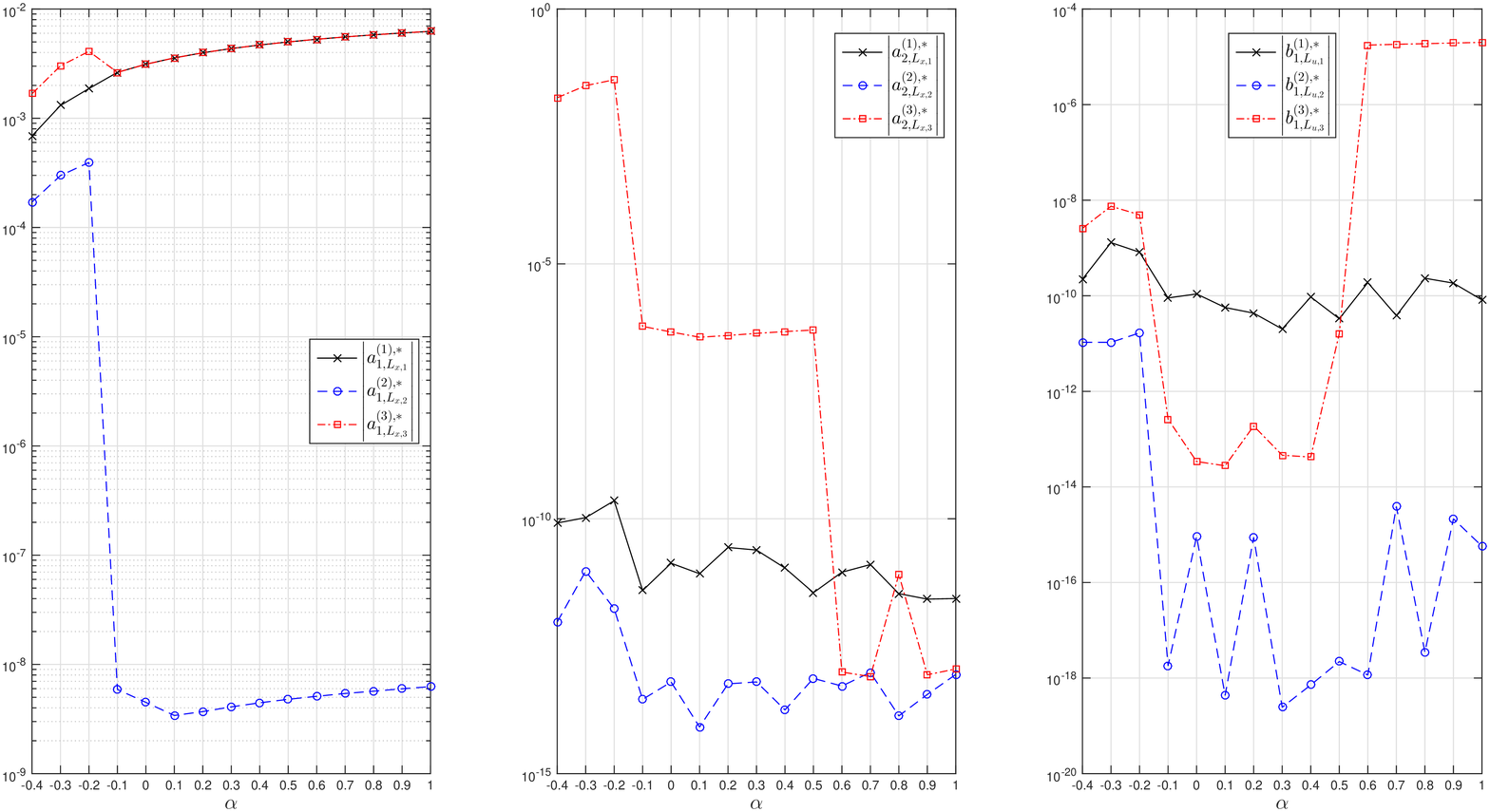}
\caption{The figure shows the magnitudes of the last coefficients in the state and control truncated series of Example 2 on the mesh intervals $\bm{\Omega}_{1,\text{orig}} = [0,0.3], \bm{\Omega}_{2,\text{orig}} = [0.3,0.7]$, and $\bm{\Omega}_{3,\text{orig}} = [0.7,1]$ using $N_k = M_k = 4, L_{x,k} = L_{u,k} = 3$, for all $ k \in \mathbb{K}$, and $\alpha = -0.4(0.1)1$.}
\label{Ex2Coeff0p4to1}
\end{figure}
\begin{figure}[ht]
\centering
\includegraphics[scale=0.6]{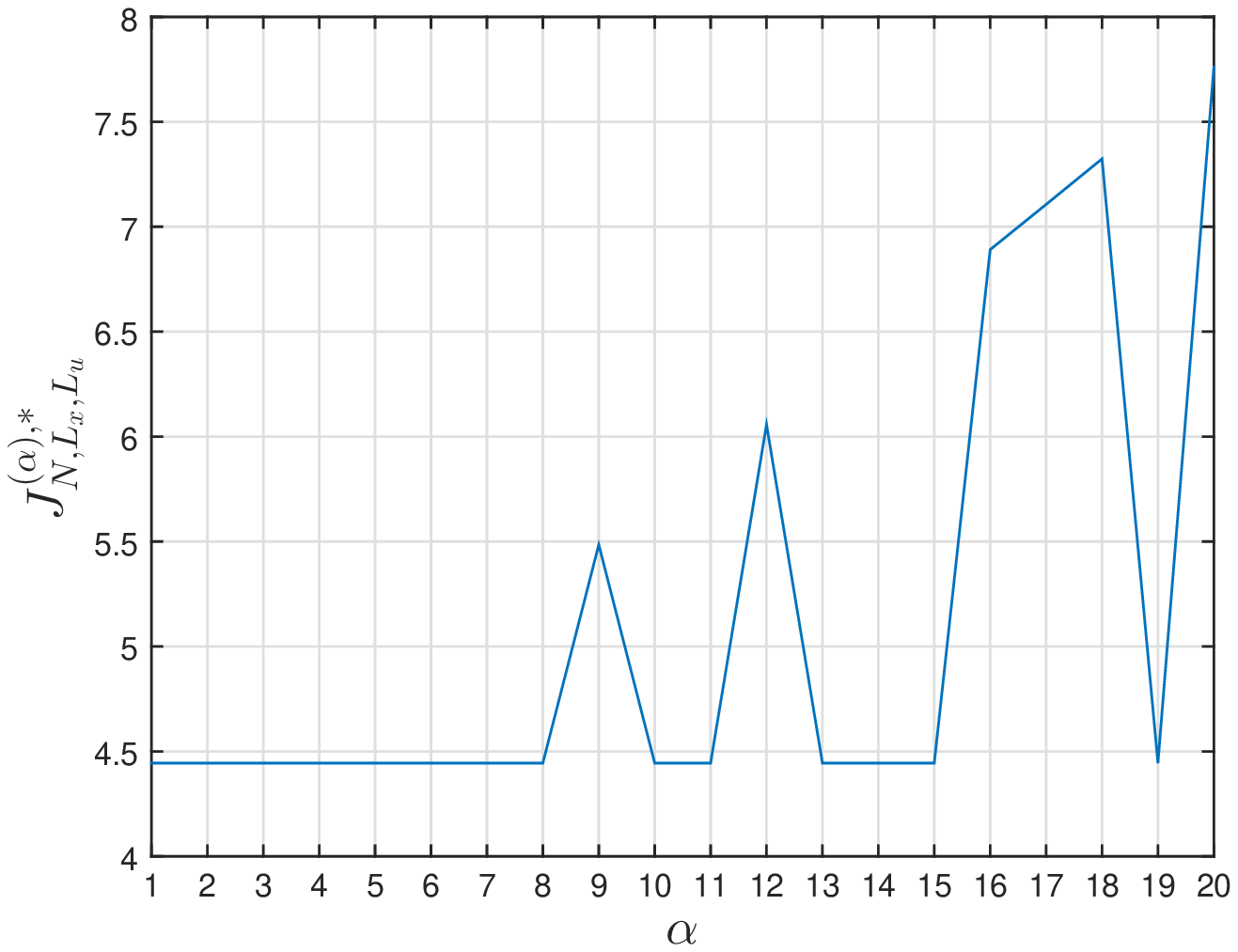}
\caption{The figure shows the plot of the approximate optimal cost functional $J^{(\alpha),*}_{N,L_x,L_u}$ of Example 2 for $N_k = M_k = 4, L_{x,k} = L_{u,k} = 3$, for all $ k \in \mathbb{K}$, and $\alpha = 1(1)20$.}
\label{Ex2Cost1to20}
\end{figure}
\begin{figure}[ht]
\centering
\includegraphics[scale=0.35]{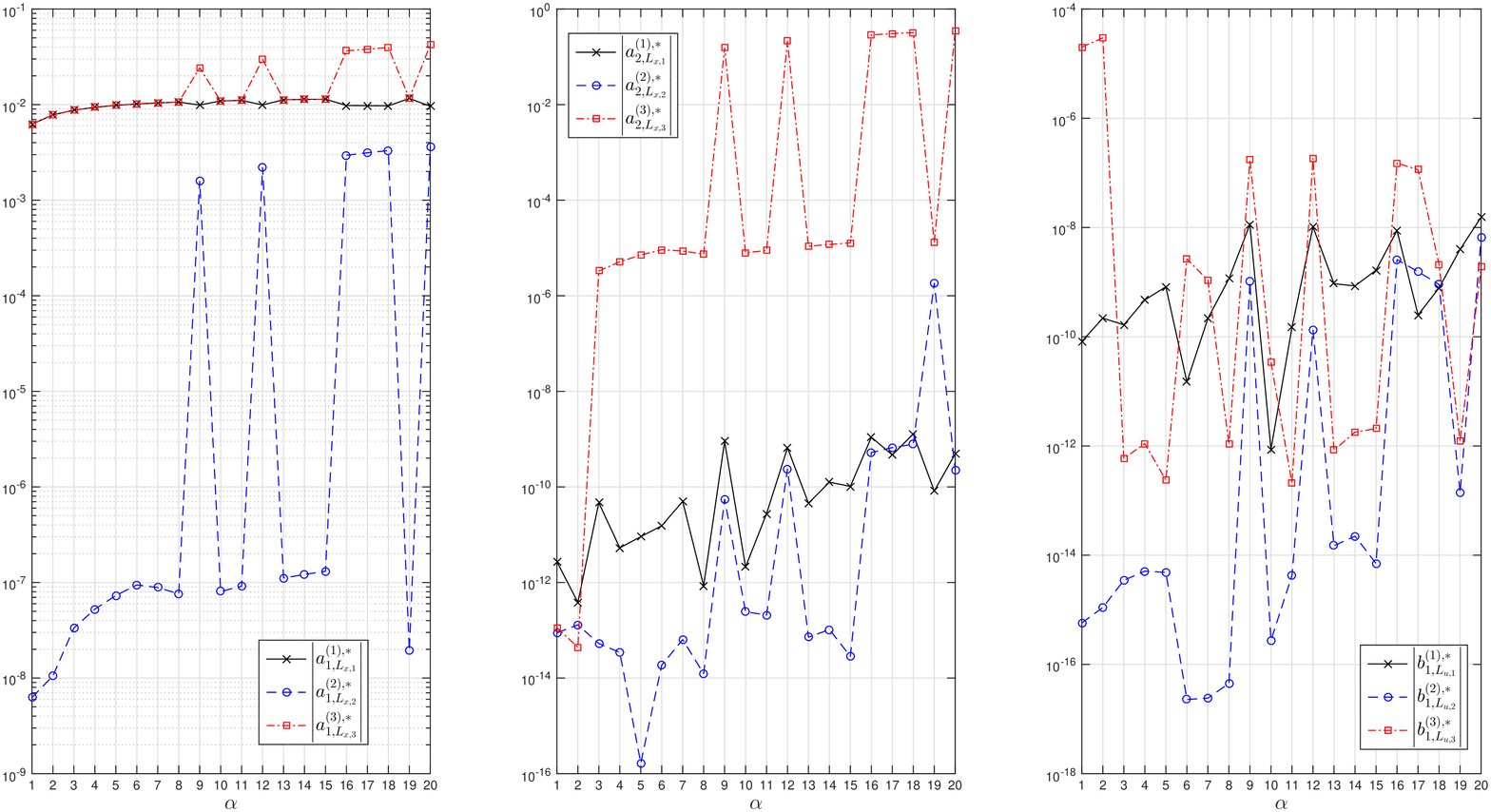}
\caption{The figure shows the magnitudes of the last coefficients in the state and control truncated series of Example 2 on the mesh intervals $\bm{\Omega}_{1,\text{orig}} = [0,0.3], \bm{\Omega}_{2,\text{orig}} = [0.3,0.7]$, and $\bm{\Omega}_{3,\text{orig}} = [0.7,1]$ using $N_k = M_k = 4, L_{x,k} = L_{u,k} = 3$, for all $ k \in \mathbb{K}$, and $\alpha = 1(1)20$.}
\label{Ex2Coeff1to20}
\end{figure}
\begin{figure}[ht]
\centering
\includegraphics[scale=0.6]{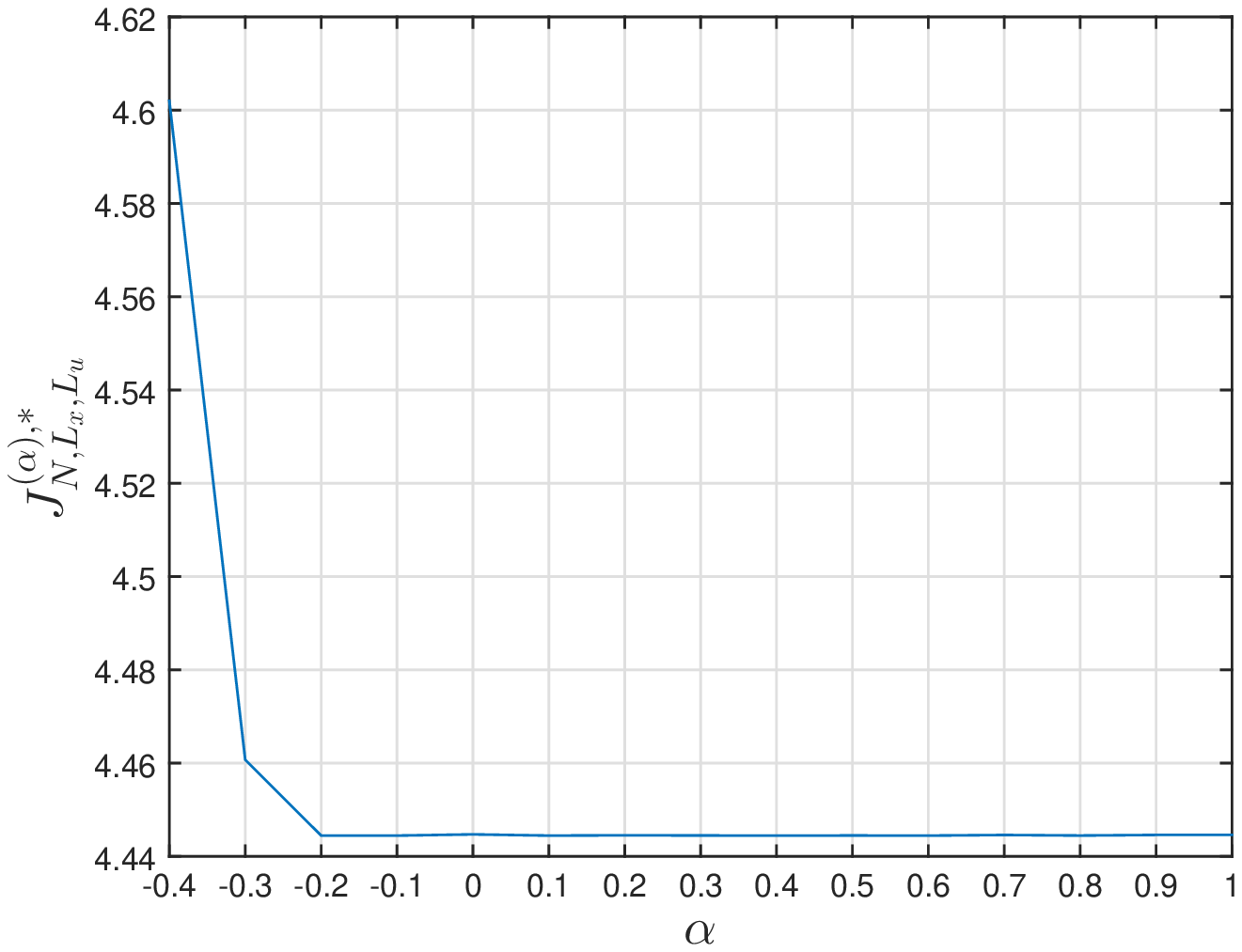}
\caption{The figure shows the plot of the approximate optimal cost functional $J^{(\alpha),*}_{N,L_x,L_u}$ of Example 2 for $N_k = M_k = L_{x,k} = L_{u,k} = 14$, for all $ k \in \mathbb{K}$, and $\alpha = -0.4(0.1)1$.}
\label{Ex2Costn0p4to122}
\end{figure}
\begin{figure}[ht]
\centering
\includegraphics[scale=0.35]{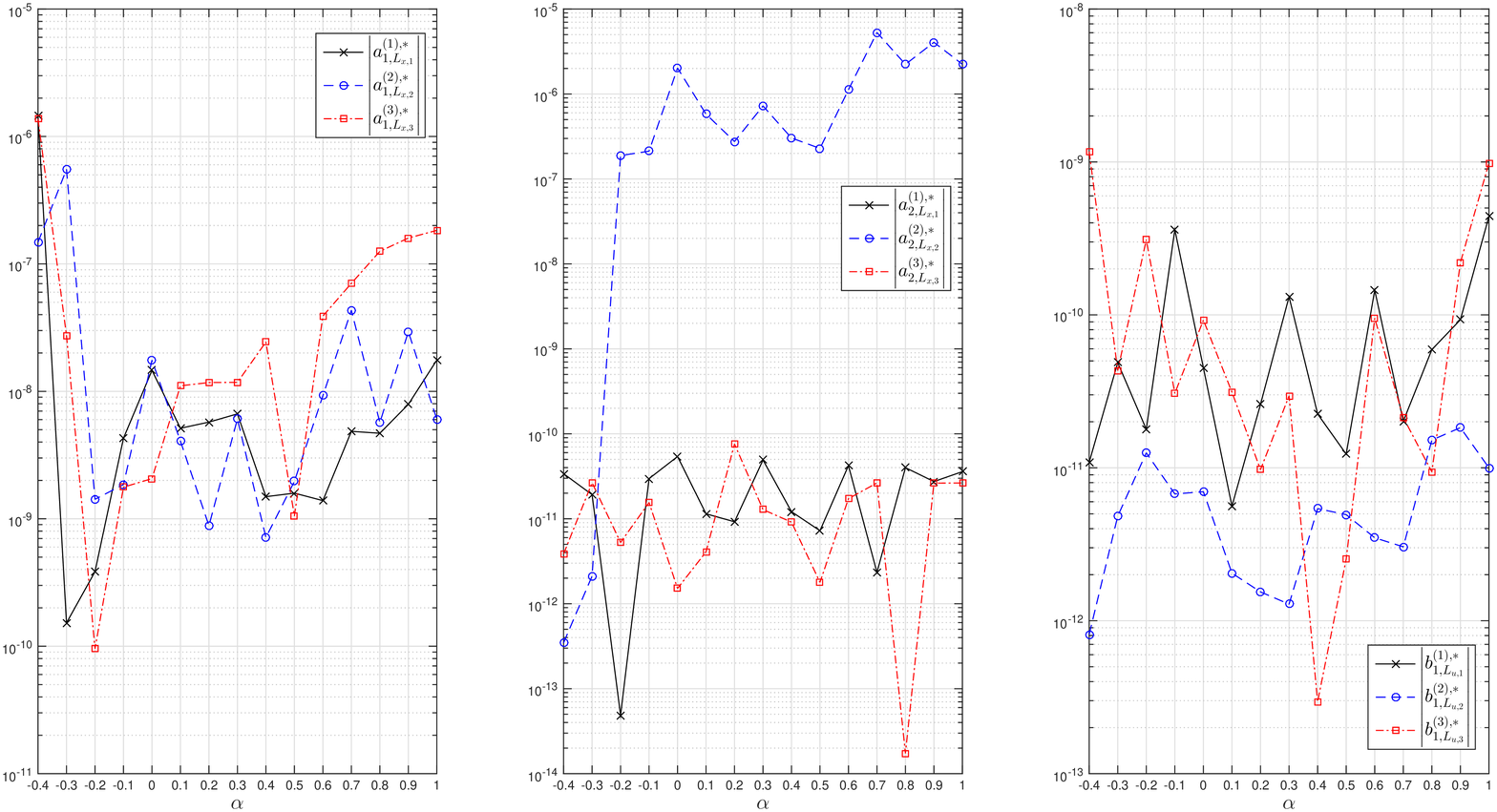}
\caption{The figure shows the magnitudes of the last coefficients in the state and control truncated series of Example 2 on the mesh intervals $\bm{\Omega}_{1,\text{orig}} = [0,0.3], \bm{\Omega}_{2,\text{orig}} = [0.3,0.7]$, and $\bm{\Omega}_{3,\text{orig}} = [0.7,1]$ using $N_k = M_k = L_{x,k} = L_{u,k} = 14$, for all $ k \in \mathbb{K}$, and $\alpha = -0.4(0.1)1$.}
\label{Ex2Coeff0p4to122}
\end{figure}
\begin{figure}[ht]
\centering
\includegraphics[scale=0.6]{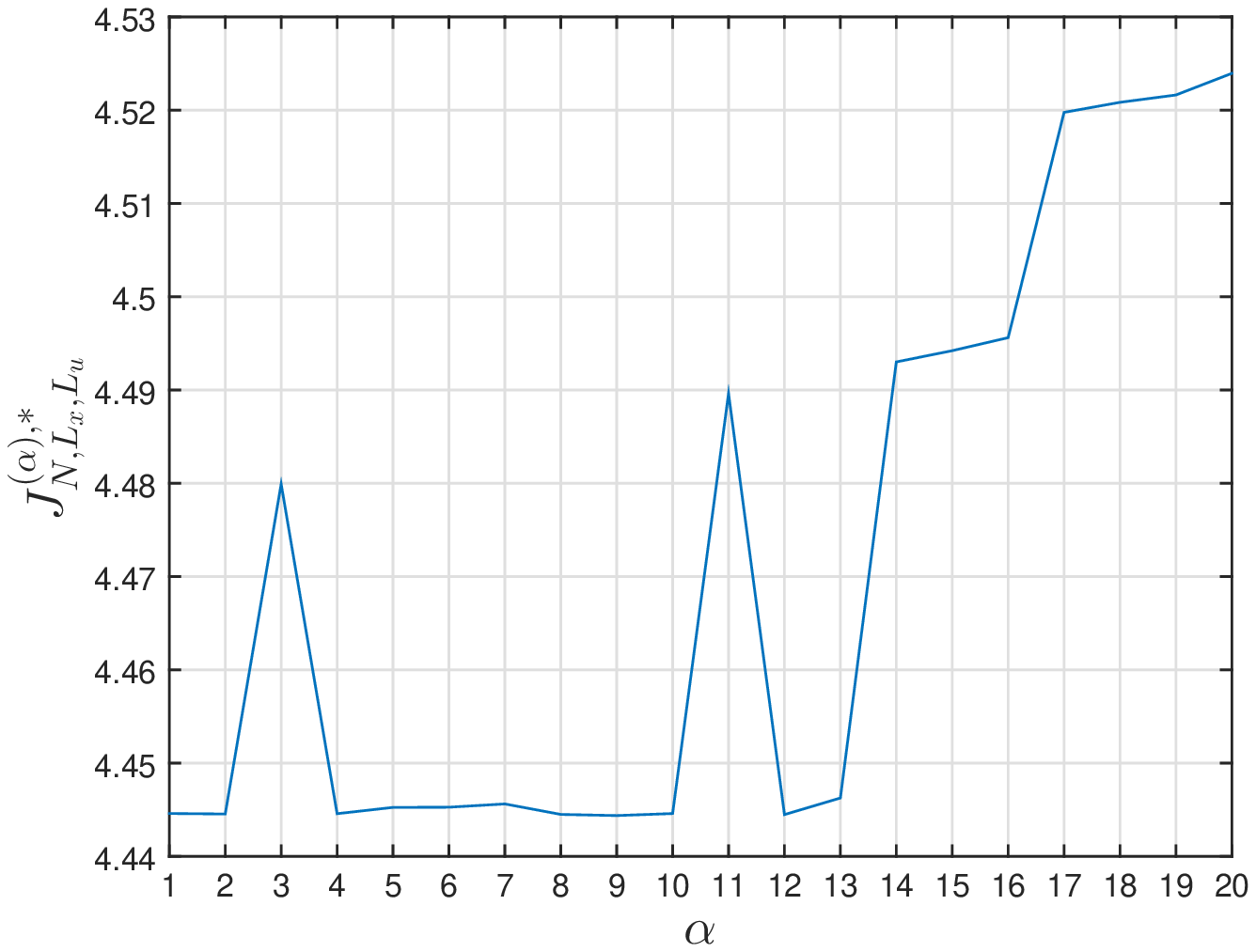}
\caption{The figure shows the plot of the approximate optimal cost functional $J^{(\alpha),*}_{N,L_x,L_u}$ of Example 2 for $N_k = M_k = L_{x,k} = L_{u,k} = 14$, for all $ k \in \mathbb{K}$, and $\alpha = 1(1)20$.}
\label{Ex2Costn0p4to12244}
\end{figure}
\begin{figure}[ht]
\centering
\includegraphics[scale=0.35]{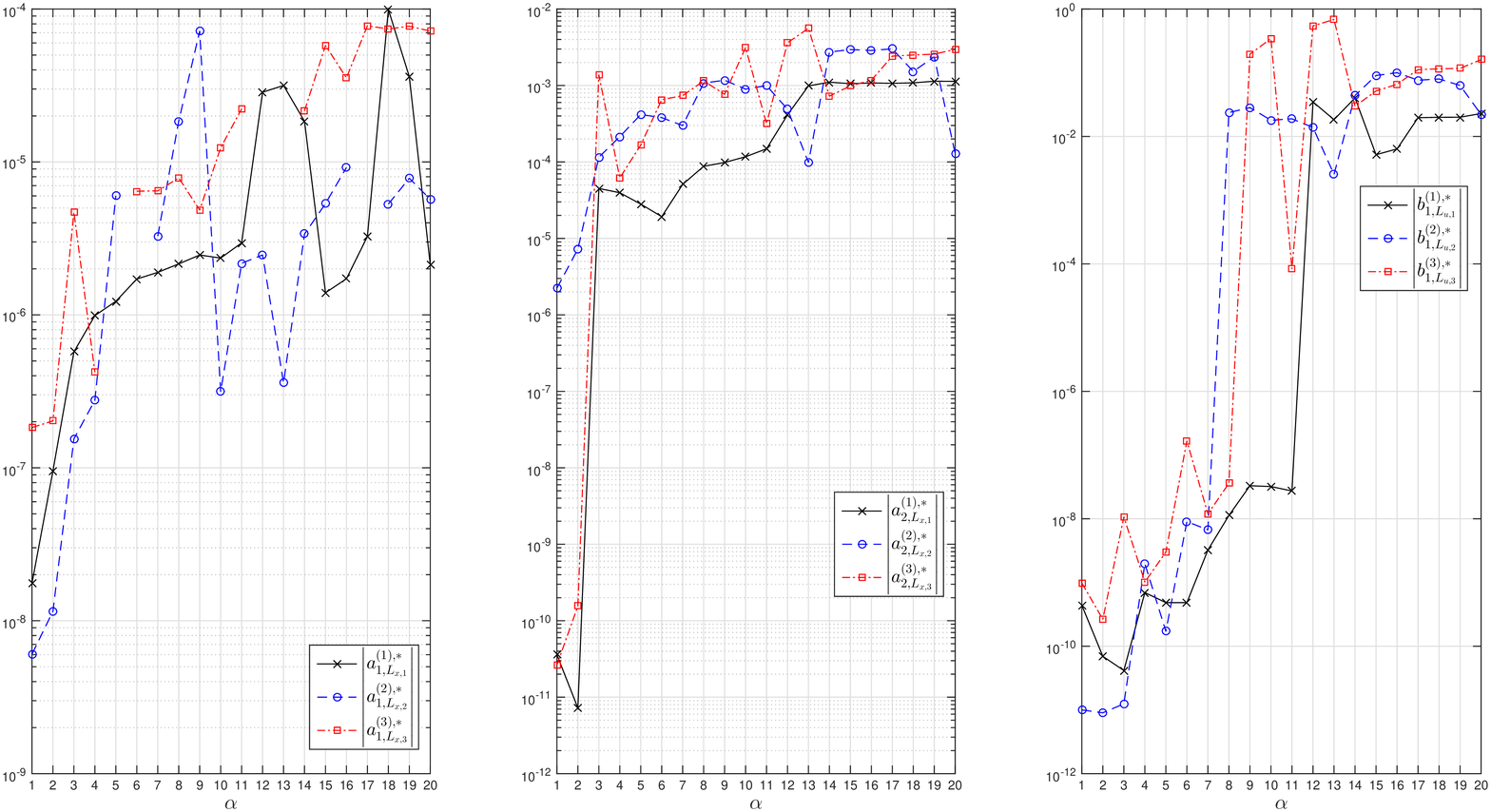}
\caption{The figure shows the magnitudes of the last coefficients in the state and control truncated series of Example 2 on the mesh intervals $\bm{\Omega}_{1,\text{orig}} = [0,0.3], \bm{\Omega}_{2,\text{orig}} = [0.3,0.7]$, and $\bm{\Omega}_{3,\text{orig}} = [0.7,1]$ using $N_k = M_k = L_{x,k} = L_{u,k} = 14$, for all $ k \in \mathbb{K}$, and $\alpha = 1(1)20$.}
\label{Ex2Coeff0p4to12244}
\end{figure}
\begin{figure}[ht]
\centering
\includegraphics[width=18cm,height=15cm]{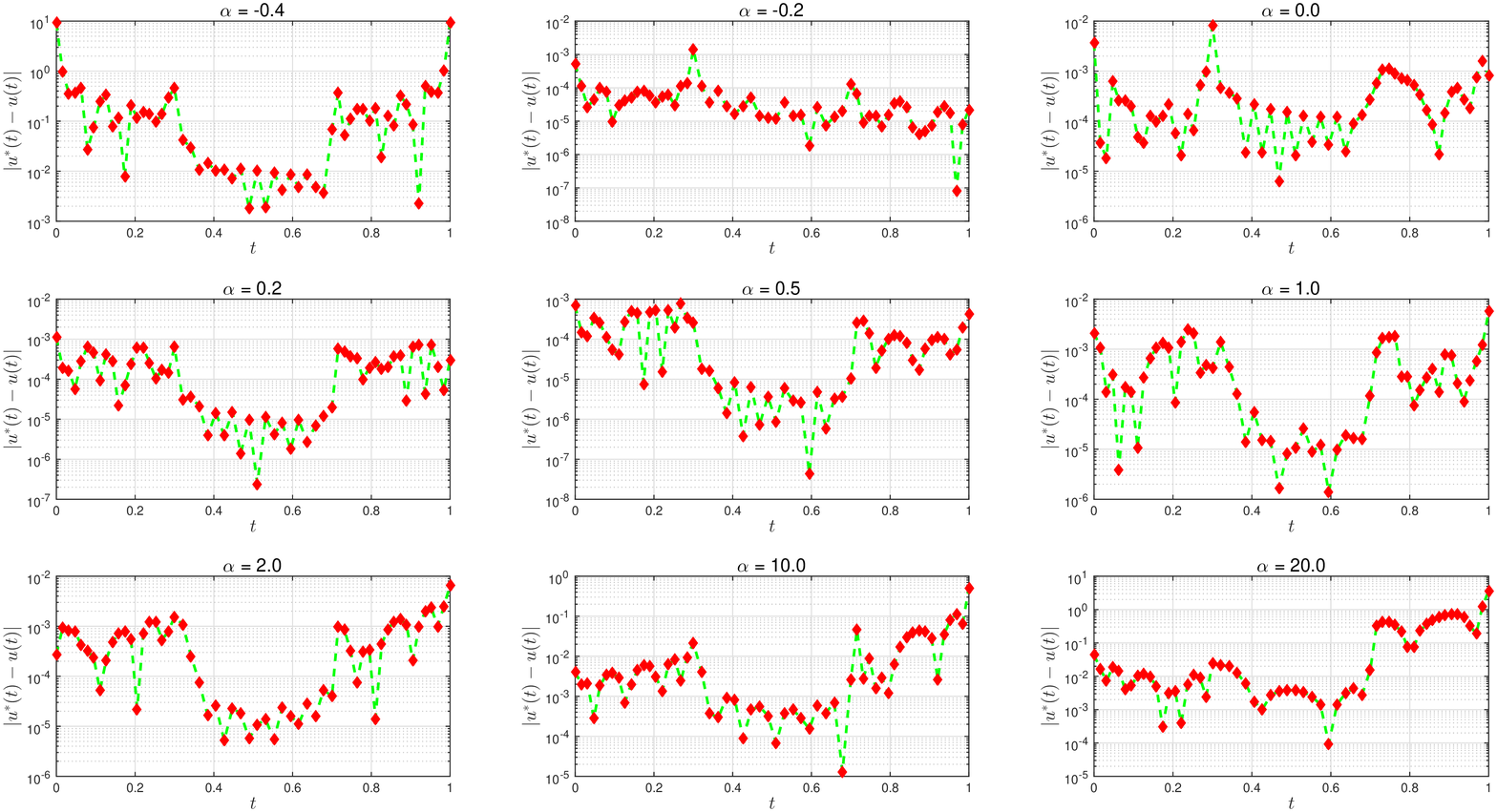}
\caption{The absolute errors of the control function values, $\left| {{u^*}(t) - u(t)} \right|$, of Example 2 on the interval $[0,1]$ in log-lin scale using $N_k = M_k = L_{x,k} = L_{u,k} = 14$, for all $ k \in \mathbb{K}$, and $\alpha = -0.4(0.2)0.2,0.5,1,2,10,20$. All plots were generated using $20$ linearly spaced nodes in each domain $\bm{\Omega}_{k,\text{orig}}$, for all $ k \in \mathbb{K}$.}
\label{fig:All}
\end{figure}

\paragraph{\textbf{Example 3}} Consider the following nonsmooth nonlinear optimal control problem:
\begin{subequations}\label{Ex3:3}
\begin{equation}
	{\text{Minimize }}J = \int_0^2 {u(t) \left( {u(t) - t} \right)\,dt}
\end{equation}
	\text{subject to } 
\begin{equation}
		\dot x(t) =  - \left| {x(t) - 0.5} \right| + \frac{{2\,\left( {u(t) + 1} \right)}}{{t + 2}} - 0.5,
\end{equation}
\begin{equation}
		x(0) = 0.1,
\end{equation}
\begin{equation}
		0 \le x(t) \le 1\quad \forall t \in [0,2],
\end{equation}
\begin{equation}
		- 1 \le u(t) \le 1\quad \forall t \in [0,2].
\end{equation}
\end{subequations}
The exact optimal control and optimal cost functional value of this problem are $u^*(t) = t/2$, for all $ t \in [0,2]$ and $J^* = -2/3$, respectively; cf. \cite{Skandari2016}. This problem was numerically solved very recently by \cite{Skandari2016} using a generalized derivative, which can transform the nonsmooth problem into another smooth optimal control problem under certain assumptions. The smooth problem was then discretized using a standard Chebyshev pseudospectral method integrated with classical Chebyshev differentiation matrices for approximating derivatives and Clenshaw-Curtis quadrature formula to approximate the cost functional. The reported approximate optimal cost functional value was $-0.666666610$ using $21$ Chebyshev-Gauss-Lobatto nodes-- an approximation that is accurate to $7$ significant digits.

We implemented the present GSE for solving the problem numerically using the parameter settings $N_k = 5, L_{x,k} = L_{u,k} = 6, M_k = 16, \bar M_k = 6, \bar N_k = \bar L_{x,k} = \bar L_{u,k} = 2, N_{k,\max} = L_{x,k,\max} = L_{u,k,\max} = 30\; \forall k \in \mathbb{K}, \alpha = 0, \epsilon_{\mathbf{R}} = 10^{-3}, \epsilon_{\text{coeff}} = 10^{-4}, \rho = 2, k_{\max} = 20$, and $\epsilon_{ES} = 0.2$. All state and control coefficients were initially set to zero. The computational algorithm breaks the transformed interval $[-1,1]$ into the following four domains $\bm{\Omega}_1 = [-1, 0.2361], \bm{\Omega}_2 = [0.2361,0.4721], \bm{\Omega}_3 = [0.4721,0.7082]$, and $\bm{\Omega}_4 = [0.7082,1]$; thus $K = 4$. These four mesh intervals correspond to the four domains $\bm{\Omega}_{1,\text{orig}} = [0, 1.236], \bm{\Omega}_{2,\text{orig}} = [1.236,1.472], \bm{\Omega}_{3,\text{orig}} = [1.472,1.708]$, and $\bm{\Omega}_{4,\text{orig}} = [1.708,2]$ of the original optimal control problem. Figure \ref{Ex3_All} shows the plots of the approximate state function, $x(t)$, exact and approximate control functions, $u^*(t)$ and $u(t)$, respectively, on the interval $[0,2]$. The corresponding approximate optimal cost function value rounded to 12 decimal digits is $J^{(\alpha),*}_{N,L_x,L_u} \approx -0.666666666648$-- an approximation that is accurate to $10$ significant digits using merely $7$ shifted Gegenbauer coefficients for the state and control trajectories and $6$ shifted Gegenbauer-Gauss collocation points in each element. Figure \ref{Ex3_u_err} shows the absolute error of the control function values, $\left| {{u^*}(t) - u(t)} \right|$, on the interval $[0,2]$ in log-lin scale, where again we can clearly observe an exponential error decay. In contrast, Figure \ref{Ex3_u_err_single} shows the slow rate of convergence associated with the absolute error on the interval $[0,2]$ in log-lin scale using the same parameter settings, but with a single collocation grid. 

\begin{figure}[ht]
\centering
\includegraphics[scale=0.7]{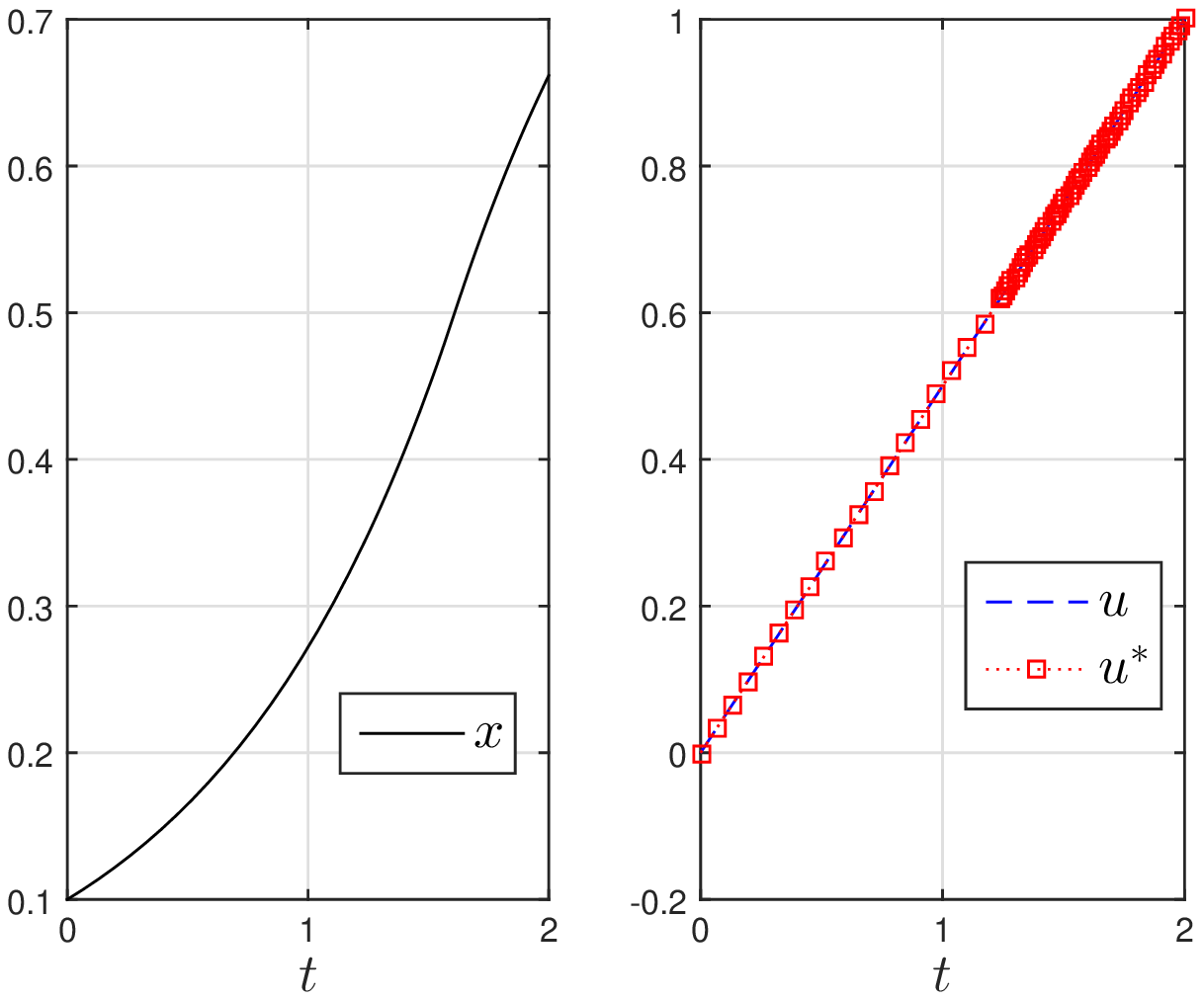}
\caption{The figure shows the plots of the approximate state function, $x(t)$ (left), exact and approximate control functions, $u^*(t)$ and $u(t)$ (right), respectively, of Example 3 on the interval $[0,2]$ using the initial values $N_k = 5, L_{x,k} = L_{u,k} = 6$, for all $ k \in \mathbb{K}$, and $\alpha = 0$. The plots were generated using $20$ linearly spaced nodes in each domain $\bm{\Omega}_{k,\text{orig}}$, for all $ k \in \mathbb{K}$.}
\label{Ex3_All}
\end{figure}
\begin{figure}[ht]
\centering
\subfigure[]{
\includegraphics[scale=0.525]{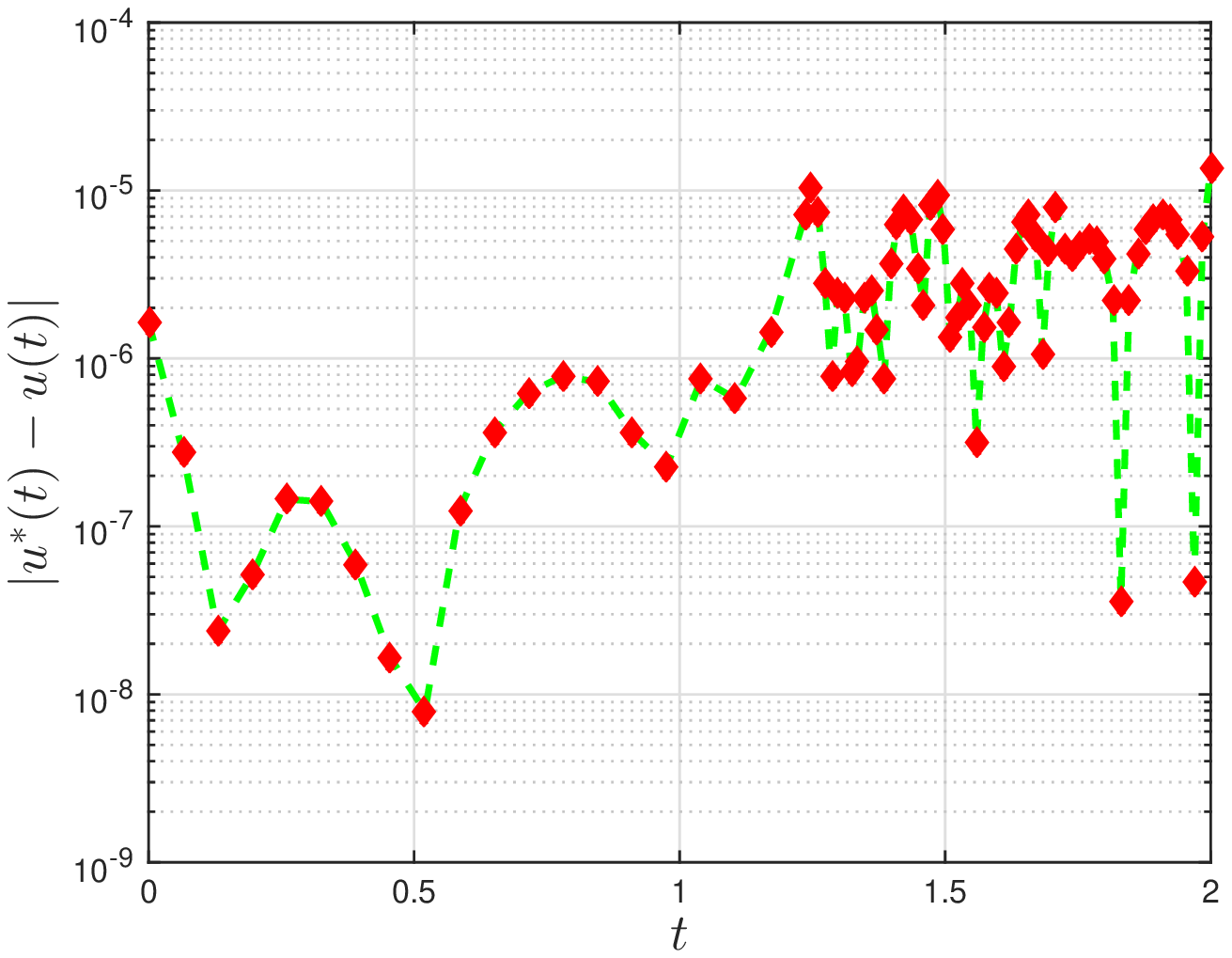}
\label{Ex3_u_err}
}
\subfigure[]{
\includegraphics[scale=0.525]{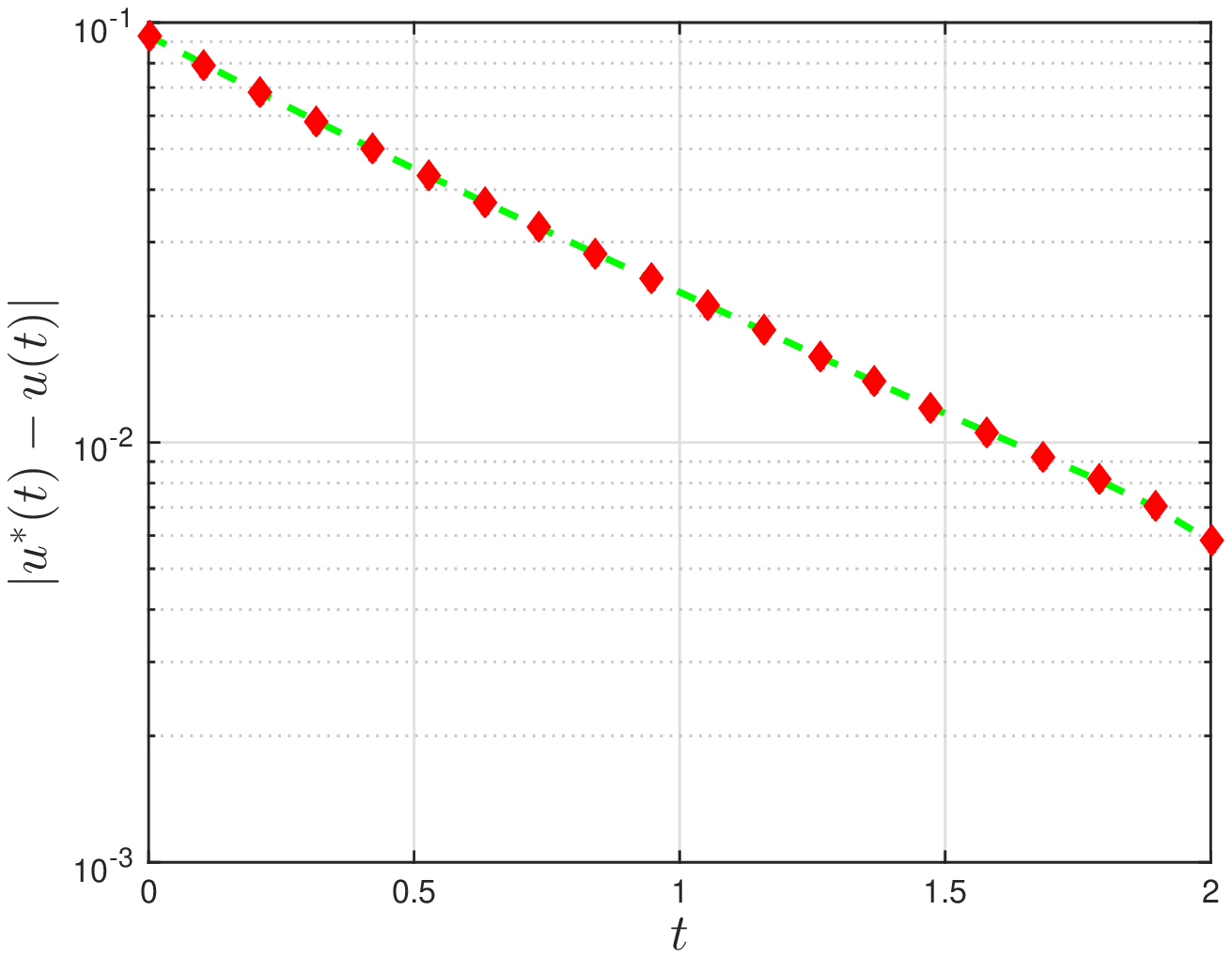}
\label{Ex3_u_err_single}
}
\caption{Figure \subref{Ex3_u_err} shows the absolute error of the control function values, $\left| {{u^*}(t) - u(t)} \right|$, of Example 3 on the interval $[0,2]$ in log-lin scale using the initial values $N_k = 5, L_{x,k} = L_{u,k} = 6$, for all $ k \in \mathbb{K}$, and $\alpha = 0$. The plot was generated using $20$ linearly spaced nodes in each domain $\bm{\Omega}_{k,\text{orig}}$, for all $ k \in \mathbb{K}$. Figure \subref{Ex3_u_err_single} shows the absolute error of the control function values $\left| {{u^*}(t) - u(t)} \right|$ on the interval $[0,2]$ in log-lin scale using a single collocation grid. The plot was generated using $20$ linearly spaced nodes in $[0,2]$.}
\label{fig:Ex3specific1}
\end{figure}
\section{Limitations}
The current GISE method was tested on only three numerical test problems in an attempt to reduce the size of the manuscript. However, further test problems may be necessary to verify further the power of the proposed method. Moreover, a further theoretical study may be conducted to analyze the convergence of the GISE method. 
\section{Conclusion}
\label{conc}
Motivated by the spectral accuracy offered by spectral element methods, we have proposed a fast, economic, and high-order algorithm for the solution of nonlinear optimal control problems exhibiting smooth/nonsmooth solutions. The coalition of information derived from the residual of the discrete dynamical system equations and the magnitude of the last coefficients in the state and control truncated series forms a powerful multicriterion adaptive strategy to boost the accuracy of the state and control approximations. Another major source for the strength of the proposed method lies in the free rectangular form of the KESOBGIM, which allows for excellent approximations to integrals with accuracy approaching machine precision. Remarkably, this significant result is achieved regardless of the number of collocation points used in the discretization process. The numerical experiments support collocations of nonlinear optimal control problems performed at $\mathbb{S}_{N_k}^{(\alpha^{(k)})}: \alpha^{(k)} \in I_{\varepsilon,r}^G$, for all $ k \in \mathbb{K}$, for small/medium numbers of collocation points and Gegenbauer expansion terms. The proposed method can be easily extended to different problems and applications.

\end{document}